\documentclass[11pt, a4paper]{article}
\usepackage{amssymb,amsmath,amsfonts,amsthm, amsbsy}  
\usepackage{hyperref}

\DeclareMathOperator{\Div}{div}
\DeclareMathOperator{\DIV}{Div}
\DeclareMathOperator{\Dom}{Dom}
\DeclareMathOperator{\Ric}{Ric^{\#}}

\newcommand{\D}{{\rm I \! D}} 
\newcommand{\E}{{\mathbb E}} 
\newcommand{\F}{{\mathcal F}} 
 
\newcommand{\HH}{{\mathcal H}} 
\newcommand{\I}{{\mathcal I}} 
 
\newcommand{\LL}{{\mathcal L}} 
\newcommand{\R}{{{\mathbf R} }}
\newcommand{\pnabla}[1][]{{\nabla\!\!\!\!\!\!\nabla^{\,#1}}}
 
\newcommand{\RR}{{\mathcal R}} 
\newcommand{\pathR}{{\mathbb{R}}} 
\newcommand{\pathT}{{\mathbb{T}}}
\newcommand{\V}{{\mathcal V}}
\newcommand{\W}{{\mathcal W}} 
\newcommand{\WW}{{\mathcal W}} 

\def\Ric{{\mathop {\rm Ric} }}
\def\div{{\mathop{\rm div}}}
\def\ker{{\mathop{\rm Ker}}}
\def\image{{\mathop{\rm Image}}}
\def\id{{\mathop {\rm Id}}}
\def\s.t.{{\mathop {\rm s.t.}}}
\def\ev{{\mathop {\rm ev}}}
 
\def\paral{/\kern-0.55ex/} 
\def\parals_#1{/\kern-0.55ex/_{\!#1}} 
\def\n#1{|\kern-0.24em|\kern-0.24em|#1|\kern-0.24em|\kern-0.24em|} 

\newtheorem{theorem}{Theorem}[section] 
\newtheorem{proposition}[theorem]{Proposition} 
\newtheorem{lemma}[theorem]{Lemma} 
\newtheorem{corollary}[theorem]{Corollary}

\theoremstyle{definition}
\newtheorem{remark}[theorem]{Remark}

\begin{document}
\title{The vanishing of $L^2$ harmonic one-forms on based path spaces} 
\author{K. D. Elworthy\footnote{Mathematics Institute, University of Warwick, Coventry, CV4 7AL, UK}
and Y. Yang\footnote{Department of Mathematics, University of Bristol, Bristol,  BS8 1TW, UK}} 
\maketitle  
\begin{abstract} 
We prove the triviality of the first $L^2$ cohomology class of based path spaces of Riemannian manifolds furnished with Brownian motion measure, and the consequent vanishing of $L^2$ harmonic one-forms.  
We give explicit formulae for closed and co-closed one-forms expressed as differentials of functions and co-differentials of $L^2$ two-forms, respectively; these are considered as extended Clark-Ocone formulae. 
A feature of the proof is the use of the temporal structure of path spaces to relate a rough exterior derivative operator on one-forms to the exterior differentiation operator used to construct the de Rham complex and the self-adjoint Laplacian on $L^2$ one-forms. This Laplacian is shown to have a spectral gap.
\end{abstract} 
\textbf{Keywords:} 
$L^2$ cohomology, Hodge decomposition, $L^2$ harmonic forms, path space, Banach manifold, Wiener measure, Malliavin calculus, Markovian connection, Clark-Ocone formula, isometry subspace, spectral gap.
\\
\textbf{MSC2010:} 
58J65, 60H07, 58A12, 58A14.
\\
\textbf{Acknowledgements:} 
The second author was partially funded by an Early Career Fellowship while at the Warwick Institute of Advanced Study.
\section{Introduction}
\paragraph{A.}
It is a well-known classical theorem that a smooth vector field $V$ on $\R^n$ is a gradient if and only if its derivative $DV(x)\in \LL(\R^n;\R^n)$  is symmetric at each point $x$.  This is equivalent to saying that a smooth differential one-form $\phi:\R^n\to (\R^n)^*$ is the derivative of a real-valued function if and only if the derivative  of $\phi$ at $x$,  $D\phi(x)\in \LL(\R^n;\LL(\R^n;\R)) \simeq \LL(\R^n,\R^n;\R)$,
 is symmetric. If the bilinear map $D\phi(x)$ is considered as an element of $\LL(\otimes^2\R^n;\R)$, this is equivalent in turn to the vanishing of $D\phi(x)$ on the subspace $\wedge^2\R^n$ of skew-symmetric two-tensors.
\par
For a smooth manifold $M$, the corresponding condition on a differential one-form $\phi$ is invariantly expressed by the vanishing of the exterior derivative $d^1\phi:\wedge^2TM\to \R$
 given by, for any two vector fields $V_1, V_2$ on $M$, 
\begin{equation}
\label{eq:Palais}
2(d^1\phi)_x(V_1(x)\wedge V_2(x))
=L_{V_1}(\phi(V_2))(x)-L_{V_2}(\phi(V_1))(x)
-\phi\left( [V_1,V_2](x)\right),
\end{equation}
where $L_V$ denotes Lie differentiation in the direction $V$.
\par
This condition is only necessary in general. The first de Rham cohomology  group $H^1(M;\R)$ measures the extent to which it fails to be sufficient: 
\begin{equation}
\label{eq:H^1}
H^1(M;\R):= \frac{\ker(d^1)}{\image(d)}, 
\end{equation}
where $d$ refers to the usual differentiation of functions on $M$.
\par
It is immediate that formula (\ref{eq:Palais}) agrees with the definition of $d^1\phi$ as the anti-symmetrisation of the covariant derivative 
\[
\nabla\phi\in \LL(TM;\LL(TM;\R)) \simeq \LL(TM,TM;\R)
\] 
for any torsion-free connection.  
For an arbitrary connection we have
\begin{equation} 
\label{eq:extder}
2(d^1\phi)(v_1\wedge v_2)= (\nabla_{v_1}\phi)(v_2)-(\nabla_{v_2}\phi)(v_1)
+\phi (T(v_1,v_2)),  v_1,v_2\in T_xM, 
\end{equation}
where $x\in M$ and $T:TM\oplus TM\to TM$ is the torsion tensor.
\par
Note our convention of having the factor $2$ in these definitions. This is in agreement with \cite{elworthy2008l2} \cite{elworthy2010stochastic} \cite{yang2011thesis}  and Kobayashi and Nomizu \cite{kobayashi1969foundations}. It is essentially forced by our wish to treat
the spaces of exterior powers of any vector space as subspaces of the corresponding tensor product spaces, 
but the Hodge-Kodaira Laplacian will no longer have the usual well known form.
\paragraph{B.}
The above makes sense for  general  Banach manifolds $\mathcal{M}$ (see, e.g., \cite{lang2002introduction}), though suitable completions need to be taken for the tensor products. In fact, when $\mathcal{M}$ is infinite dimensional, we let $\otimes^2T_x\mathcal{M}$ denote the completed tensor product using the largest cross norm, i.e., the projective tensor product, and similarly for its subspace $\wedge^2T_x\mathcal{M}$. Then two-forms are sections of the dual bundle to $\wedge^2T\mathcal{M}$.
\par
When $G$ and $H$ are Hilbert spaces, we use $G\otimes H$ for the standard Hilbert space completion of the tensor products, so there is the natural isometry $G\otimes H\to \LL_2(H;G)$ onto the space of Hilbert-Schmidt operators. 
\paragraph{C.} 
When our manifold has a suitable Borel measure and a Riemannian metric 
(or a given smooth family of norms on its tangent spaces), it makes sense to  consider forms which are in $L^2$.  In finite dimensions, if the manifold $M$ is complete Riemannian and the measure smooth, the exterior differentiation on smooth forms is closable, leading to a closed operator which we still write as $d^1$, from a dense domain in $L^2\Gamma T^*M$, the space of $L^2$ sections of the cotangent bundle, to $L^2\Gamma (\wedge^2TM)^*$.
\par
The first $L^2$ cohomology group is then the vector space $L^2H^1(M;\R)$ defined by equation (\ref{eq:H^1}) but using only $L^2$ forms. There is also a Hodge decomposition
\begin{equation}
\label{eq:Hodge }
L^2\Gamma T^*M= \overline{\image(d)}\bigoplus \mathbf{H}^1(M)\bigoplus \overline{\image(d^{1*})}, 
\end{equation}
where $d^{1*}$ denotes the adjoint of $d^1$, the overlining denotes the (topological) closure, and $\mathbf{H}^1(M)$ denotes the space of \emph{harmonic} one-forms, the intersection of the kernel of $d^1$ with that of $d^*$.
\paragraph{D}
The space $C_{x_0}:= C_{x_0}([0,T]; M)$ of continuous paths from a fixed interval $[0,T]$ into a compact Riemannian manifold $M$, starting at a given point $x_0$ of $M$, has a natural smooth Banach manifold structure \cite{eellsgeometry}.  The tangent space $T_\sigma C_{x_0}$  at a point $\sigma$ consists of continuous paths $v:[0,T]\to TM$ lying over $\sigma$, with $v(0)=0\in T_{x_0}M$.  Differential forms on this manifold as described above will be called \emph{geometric} forms, to distinguish them clearly from the $H$-forms which we describe below.  Examples are the smooth cylindrical forms, which are the pull backs $\ev_{\underline{t}}^*\psi$ of smooth forms $\psi$ on the product manifold $M^{\underline{t}}$ for arbitrary finite subsets ${\underline{t}}$ of $[0,T]$, where $\ev_{\underline{t}}:C_{x_0}\to M^{\underline{t}}$ is the restriction map (essentially the multiple evaluation map). 
\par
Furnish $C_{x_0}$ with its Brownian motion measure, say, $\mu$. Following  Gross \cite{gross1967potential} it is standard to have a differential calculus based on differentiation in the directions of certain Hilbert spaces. For this we use the ``Bismut" subspaces $\HH_\sigma $ of $T_\sigma C_{x_0}$, defined for each path $\sigma$ by 
\[ 
\HH_\sigma=\left\{v\in T_\sigma C_{x_0}:v_t= \parals_t^\sigma h_t, t\in[0,T],\text{ some } h\in L^{2,1}_0([0,T];T_{x_0}M)\right\} .
\]
Here $\parals_t^\sigma:T_{x_0}M\to T_{\sigma_t}M$ denotes parallel translation along $\sigma$ using the Levi- Civita connection of $M$. Note that such objects are only defined for almost all paths.  
These $ \HH_\sigma$ become Hilbert spaces continuously included in the geometric tangent spaces $T_\sigma C_{x_0}$ under the  \emph{damped} inner products $\langle   -, -\rangle_\sigma$; see below for more details or \cite{elworthy2006geometric} for an overview. We let $L^2\Gamma \HH$, $L^2\Gamma \HH^*$, etc., denote the spaces of $L^2$ $H$-vector fields and $L^2$ $H$-one-forms, respectively; in other words, the $L^2$ sections of the relevant ``bundles".  As usual in Malliavin calculus, $H$-differentiation of $G$-valued cylindrical functions, for $G$ a separable Hilbert space, extends to give a closed densely defined operator $d$ from its domain in $L^2(C_{x_0}; G)$ to $L^2\Gamma \LL_2(\HH;G)$. Denote its domain by $\D^{2,1}G$, or simply $\D^{2,1}$ when $G=\R$. The corresponding gradient operator, $f\mapsto \nabla f$, has the same domain with values in $G \otimes L^2\Gamma \HH $,  or just $L^2\Gamma\HH$ when $G=\R$.
\par
There is a natural connection on $\HH$, the damped Markovian connection,  giving a closed covariant derivative operator $\pnabla$, from the domain $\D^{2,1}\HH$ in $L^2\Gamma\HH$ to $L^2\Gamma \LL(\HH;\HH)$; see Section \ref{sec:Ito map} below.  Its torsion tensor $\pathT$ will be considered as an $L^p$ section,  for all $1\le p<\infty$, of the bundle of  continuous skew-symmetric bilinear maps $\LL_{skew}(\HH,\HH; TC_{x_0})$ from the Bismut tangent bundle to the tangent bundle of $C_{x_0}$;  see Appendix B of \cite{elworthy2008l2} but note the misprint there, where $\LL_{2}(\wedge^2\HH; TC_{x_0})$ should be $\LL_{skew}(\HH,\HH; TC_{x_0})$.  
The curvature operator, denoted by $\pathR$, will be considered as an $L^p$ section of  $\LL(\wedge^2T_\sigma C_{x_0} ; \wedge^2T_\sigma C_{x_0})$.
\par
In general we adopt the same notation as used in \cite{elworthy2008l2}. An exception is that here  $\F^{x_0}_t$ refers to the natural filtration on $C_{x_0}$ and $\F^\I_t$ to the filtration on the Wiener space generated by the solution from $x_0$ of a stochastic differential equation (SDE), defined in Section \ref{sec:Ito map} below.
\paragraph{E.} 
Following Shigekawa's work \cite{shigekawa2003vanishing} on the Hodge theory for the Wiener space, it might be hoped that an $L^2$ de Rham theory for $C_{x_0}$ could be based on ``forms" which are sections of  the dual spaces to the exterior powers $\wedge^q\HH$ of the Bismut tangent bundle. However, the Lie bracket of sections of $\HH$ may not be sections of $\HH$ in the presence of curvature (Cruzeiro and Malliavin \cite {cruzeiro1996renormalized}, Driver \cite{driver1999lie}), so the straightforward application of the usual definition of exterior derivative cannot give a closable operator acting between $L^2$ spaces of sections of these bundles (Leandre \cite {leandre2002analysis}). This can be seen from formula (\ref{eq:Palais})  for one-forms and is also shown in the following analogue of  equation (\ref {eq:extder}) for a cylindrical one-form $\phi$ on $C_{x_0}$: for almost all $\sigma\in C_{x_0}$, 
\begin{equation} 
\label{eq:pextder}
2(d^1\phi)(v_1\wedge v_2)= \pnabla_{v_1}\phi (v_2)-\pnabla_{v_2}\phi(v_1)+\phi\left(\pathT(v_1,v_2)\right), \quad v_1,v_2\in \HH_\sigma, 
\end{equation}
where the torsion does not in general take values in $\HH$.
\par
There have been many efforts to circumvent the problem; see Leandre  \cite{ leandre2002analysis} for a survey.  Elworthy and Li \cite{elworthy2000special} proposed to replace, for $q\ge 2$, the Hilbert spaces $\wedge^q \HH_\sigma$ by a family of Hilbert spaces $\HH^{(q)}_\sigma$, which are continuously included in $\wedge^q T_\sigma C_{x_0}$, while keeping the exterior derivative a closure of the classical exterior derivative on smooth cylindrical forms. An $H$-$q$-form will be a section of $\HH^{(q)*}$.   A detailed description of the case $q = 2$ is given in \cite{elworthy2008l2}, where $\HH^{(2)}$ is shown to be a deformation inside $\wedge^2 T_\sigma C_{x_0}$ of the exterior product $\wedge^2\HH$ of the Bismut tangent bundle by the curvature of the damped Markovian connection.  
The analysis in \cite{elworthy2000special, elworthy2008l2} proves the closability of exterior differentiation on the corresponding $L^2$ $H$-one-forms, 
defines a self-adjoint Hodge-Kodaira Laplacian on such $L^2$ $H$-one-forms, and establishes the resulting Hodge decomposition as in equation (\ref{eq:Hodge }), 
where ${d}^1$ is the closure of the geometrically defined exterior derivative. It holds that ${d}^1 d=0$.  In addition, by Fang's version of the Clark-Ocone formula \cite{fang1994inegalite}, described below, the image of $d$ is closed, so every cohomology class in 
\[
L^2H^1 (C_{x_0}) =\frac{\ker( d^1)}{\image(d)}
\]
has a unique representative in $\mathbf{H}^1(C_{x_0})$, the space of $L^2$ harmonic one-forms.  
\par
After introducing some notation and a few preliminary results, we prove a Clark-Ocone formula for $L^2$ $H$-one-forms in Theorem \ref{th:mnfd_CO_1} and Corollary \ref{co:manifold_CO_short} below, with a version for co-closed one-forms in Corollary \ref{cor:mnfd_co-closed}.  
This implies immediately that  $L^2H^1 (C_{x_0})=\{0\}$, so all $L^2$ harmonic forms vanish. Moreover, the image of $d^1$ is closed, the Hodge Laplacian for one-forms has a spectral gap, and we have an improved decomposition
\begin{equation}
\label{eq: path_Hodge1}
L^2\Gamma\HH^*= \image(d)\bigoplus\image({d}^{1*}).
\end{equation}
A similar vanishing theorem was given in \cite{elworthy2010stochastic} for $L^2H^1 (C_{x_0})$ when $M$ is a symmetric space, based on results from \cite{yang2011thesis}.
\subsection{Paths on groups; other approaches}
When $M$ has a Lie group structure  with bi-invariant metric, the problems mentioned above in the definition of an $L^2$ de Rham complex over its path space disappear if the Bismut tangent space is defined using the flat left or right invariant connection. The complex can be defined using exterior powers of $\HH$ and its cohomology vanishes as shown by Fang and Franchi \cite{fang1997differentiable}. Fang and Franchi  \cite{fang1997rham} also defined the complex, with a Hodge decomposition, for based loops on such a Lie group. More recently, Aida \cite{aida2011vanishing} showed that the resulting first $L^2$ cohomology group vanishes when  the group is simply connected.
\par
Aida used techniques from rough path theory combined with elements taken from Kusuoka's approach \cite{kusuoka1991rham} of considering a ``submanifold"  of the Wiener space, which is in some sense a model for the path or loop space on a given general compact Riemannian manifold. Kusuoka  \cite{kusuoka1991rham} constructed an $L^2$ Hodge theory in this context, with partial results on the computation of the $L^2$ cohomology for loop spaces. 
\par
Leandre \cite{leandre2002analysis} developed other approaches to de Rham theory on path and loop spaces furnished with Brownian motion measures. One of these was to get over the difficulty of defining the exterior derivative by interpreting terms such as $\phi([V_1,V_2])$ as Stratonovich stochastic integrals \cite{leandre1996cohomologie}. For based paths this led to vanishing cohomology; for based loops it gave rise to the usual  cohomology of the based loop space. However, his theories did not involve a Hodge Laplacian.
\section{It\^o maps and the damped Markovian connection.}
\label{sec:Ito map}
Our main tool is the  It\^o map of a suitable SDE on $M$, as a substitute for measure class preserving charts. Here we recall the notation and basic facts from \cite{elworthy1993conditional, elworthy1999geometry,elworthy2000special, elworthy2007ito, elworthy2008l2}.
\paragraph{A.}
Using the Levi-Civita connection on $M$, let $\frac{D}{dt}$ denote the usual covariant differentiation defined along almost all paths in $C_{x_0}$.  Let $\frac{\D}{dt}$ denote the damped version defined by $\frac{\D}{dt}=\frac{D}{dt}+\frac{1}{2} \Ric^{\sharp}$,  
and $W_t
: T_{x_0}M\to T_{\sigma_t}M$ the damped parallel translation defined by $\frac{\D}{dt}W_t=0$.  Here $\Ric^{\sharp}: TM\to TM$ corresponds to the Ricci curvature via $\langle \Ric^{\sharp}_x(v_1), v_2\rangle = \Ric(v_1, v_2)$, for any  $v_1,v_2\in T_xM$. 
\par
It is often convenient to use $L^2TC_{x_0}$, the $L^2$ tangent bundle of $C_{x_0}$. It is a $C^\infty$ Hilbert bundle  over $C_{x_0}$, whose fibre at a path $\sigma$ consists of measurable vector fields $V:[0,T]\to TM$  along $\sigma$ such that $\int_0^T |V_t|^2_{\sigma_t} ~dt <\infty$, with the natural inner product.  Then $\frac{\D}{dt}:\HH\to L^2TC_{x_0}$ determines an almost surely bijective map, with inverse $\W$ given by
\[
\W(V)_t=W_t\int_0^t W_s^{-1}V_s ~ds.
\]
Thus $\HH$ inherits a ``bundle" structure and a Riemannian metric, and the inner product we use for vectors $V_1,V_2$ in $\HH_\sigma$ is
\[
\int_0^T\langle\frac{\D}{ds}V_1,\frac{\D}{ds}V_2\rangle_{\sigma_s}~ds.
\]
\paragraph{B.}
Choose a smooth surjective vector bundle map $X:  \underline\R^m\to TM$  of the trivial $\R^m$-bundle over $M$ into the tangent bundle of $M$, for some $m\in \mathbb{N}$, which induces the Riemannian metric of $M$ and its Levi-Civita connection in the sense of \cite{elworthy1999geometry}. This means that, if $Y_x:T_xM\to \R^m$ is the pseudo inverse of $X(x)$ for each $x\in M$, then $Y_x=X(x)^*$ and the covariant derivative $\nabla _vU$ of a vector field $U$ in the direction of some $v\in T_xM$ is given by 
\[
\nabla _vU= X(x)d[y\mapsto Y_y U(y)](v).
\]
A basic property of this covariant derivative is (\cite{elworthy1999geometry} Proposition 1.1.1)
\begin{equation}
\label{eq:LJW_connection}
\nabla_vX(x)(e)=0, \quad \forall v\in T_xM, e\in \ker(X(x))^\perp. 
\end{equation}
\par
Let  $\{B_t\}_{t\in[0,T]}$ be the canonical Brownian motion on $\R^m$. Given $x_0\in M$, the solutions $\{x_t\}_{t\in[0,T]}$ of the SDE
\begin{equation}
\label{eq:SDE} 
dx_t=X(x_t)\circ dB_t
\end{equation}
are Brownian motions on $M$.  Denote by $C_0:=C_0([0,T];\R^m)$ the classical Wiener space, with the natural filtration $\{\F_t\}_{t\in[0,T]}$ and the Wiener measure $\gamma$. 
The It\^o map is the solution map $\I:C_0\to C_{x_0}$ of SDE (\ref{eq:SDE}), a measure-preserving map between $(C_0, \F, \gamma)$ and $(C_{x_0}, \F^{x_0}, \mu_{x_0})$, with $\I_*\gamma = \mu_{x_0}$. 
The filtration generated by $\I$ is denoted by $\{\F_t^\I\}_{t\in[0,T]}$. 
\par
For almost all $w\in C_0$, the $H$-derivative of $\I$ at $w$ can be considered as  a continuous linear map $T_w\I:H\to T_{x_.}C_{x_0}$. 
For almost all $\sigma\in C_{x_0}$ and $h\in H$, we define 
\[
\overline{T \I}_{\sigma}(h)=\E\left[T_w\I(h)|\I(w)=\sigma\right]. 
\]
In general, we denote by $\overline{f}(\sigma)$ the conditional expectation of an integrable function $f$ on $C_0$ given $\I=\sigma$, 
which gives a function on $C_{x_0}$.
For a discussion of the conditional expectation of vector bundle valued processes, see \cite{elworthy1993conditional,elworthy1999geometry}. 
\par 
Since the connection defined by the SDE (\ref{eq:SDE}) is the same as the one defining $\HH$ and its inner product, the map $\overline{T \I}_{\sigma}$ gives a projection $\overline{T \I}_{\sigma}:H\to \HH_\sigma$ for almost all $\sigma\in C_{x_0}$ (\cite{elworthy2007ito} Property 3.1) with 
\begin{equation}
\label{eq:barTI}
\overline{T \I}_{\sigma}(h)_t
=W_t\int_0^tW_s^{-1}X(\sigma_s)\dot{h}_s ds
=\W_t \left(X(x_.)(\dot{h})\right). 
\end{equation}
\par
Relatedly, we have the push-forward map $\overline{T \I(-)}_{\sigma}$ mapping any $L^2$ $H$-vector field $h$ on $C_0$ to an $\HH$-vector field $\overline{T \I(h)}$ on $C_{x_0}$, given by 
\[
\overline{T \I(h)}_{\sigma} = \E\left [T_w\I(h(w))|\I(w)=\sigma\right ], \quad \textup{ a.e. }\sigma\in C_{x_0}.
\]
This is a continuous linear map from $L^2(C_0; H)$ to $L^2\Gamma \HH$ (\cite{elworthy2000special} Theorem 2.2), 
and for $h\in \F^{x_0}_T$,  we have $h=\bar{h}\circ\I$, 
so 
\begin{equation}
\label{eq:2TI}
\overline{T \I(h)}_{\sigma}= \overline{T \I}_{\sigma}(\bar h).
\end{equation} 
By Lemma 9.2 in \cite{elworthy2008l2}, identity (\ref{eq:2TI}) also holds for an $\F_.$-adapted $H$-vector field $h$ on $C_0$ . 
\par
\paragraph {C.} Following \cite{elworthy2007ito}, we use the map $X:M\times \R^m\to TM$ in the SDE (\ref{eq:SDE}) to define
$\tilde{X}: C_{x_0} \times L^2([0,T]; \R^m)\to L^2TC_{x_0}$ by
\[(\tilde{X}(\sigma)h)_t=X(\sigma_t)(h_t), \quad \forall \sigma\in C_{x_0} , t\in [0,T], h\in L^2([0,T]; \R^m),\]
and its right inverse $\tilde{Y}_\sigma: L^2T_\sigma C_{x_0} \to L^2([0,T]; \R^m)$ by
\[\tilde{Y}_{\sigma}(k)_t=Y_{\sigma_t}(k_t),\quad \forall k\in L^2T_{\sigma}C_{x_0}.\]  
We also define $\mathbb{X}: C_{x_0} \times H
\to \HH$ by
\[
\mathbb{X}(\sigma)(h)=\overline{T \I}_{\sigma}(h)=\WW\left(\tilde{X}(\sigma)(\dot h)\right), \quad \forall \sigma\in C_{x_0} , h\in H, 
\]
with the 
right inverse $\mathbb{Y}_{\sigma}: \HH_{\sigma}\to H$ given by
\begin{equation}
\label{eq:bold_Y}
\mathbb{Y}_{\sigma}(k)_t=\int_0^t Y_{\sigma_s}(\frac{\D}{ds}k_s)ds,\quad \forall k\in \HH_{\sigma}.
\end{equation}
\par
The $L^2([0,T]; \R^m)$-valued one-form $\tilde{Y}$ induces on $C_{x_0}$ 
the pointwise connection $\tilde\nabla$, defined  for vector fields $U\in\Dom(\tilde\nabla)=\D^{2,1}(L^2TC_{x_0})$ by
\[
\tilde\nabla_v U=\tilde X(\sigma)d[\alpha\mapsto \tilde Y_{\alpha}U(\alpha)](v), \quad \sigma\in C_{x_0}, v\in T_{\sigma}C_{x_0}.
\]
The pointwise connection is metric for the $L^2$ metric, and torsion-free if $\nabla$ is chosen to be torsion-free, as is assumed here. We can use the almost surely defined map $\frac{\D}{d.}:\HH\to L^2TC_{x_0}$ to pull back $\tilde\nabla$ and obtain a metric connection, the damped Markovian connection $\pnabla$, on $\HH$:
\begin{equation}
\label{eq:connection_relation}
\pnabla= (\frac{\D}{d.})^{-1} \tilde{\nabla} \frac{\D}{d.}.
\end{equation}
Equivalently, given $U\in \D^{2,1}\HH$ and $v\in \HH_\sigma$,  for almost all  $\sigma\in C_{x_0}$, we have 
\begin{equation}
\label{eq:DMconnection}
\pnabla_vU= \mathbb{X}(\sigma) d(\mathbb{Y}U)_\sigma(v).
\end{equation}
\paragraph{D.} We also need the splitting of $\{B_t\}_{t\in 0,T]}$ into relevant and redundant noise  \cite{elworthy1993conditional, elworthy1999geometry}.  
Since $X(x_0)$ is surjective, we have the splitting 
\[
\R^m = \ker(X(x_0))^\perp\times  \ker (X(x_0)),   
\]
with independent Brownian motions
$\tilde{B}:\![0,T]\!\times\!C_0\!\to\ker (X(x_0))^\perp\subset\R^n$ and
$\beta:\![0,T]\!\times\!C_0\!\to\ker (X(x_0))\subset\R^{m-n}$, 
as described in \cite{elworthy1999geometry}, such that $\{\tilde{B}_t\}_{t\in[0,T]}$ and $\{x_t\}_{t\in[0,T]}$ have the same filtration and 
\begin{equation}
\label{eq:splitting}
dB_t= \tilde{\parals_t}^{x} d\tilde{B}_t + \tilde{\parals_t}^{x} d\beta_t,
\end{equation}
where the map $\tilde{\paral}:[0,T] \times C_{x_0} \to O(m)$ is sample continuous and adapted to $\{\F^{x_0}_t\}_{t\in [0,T]}$, with $O(m)$ being the orthogonal group of $\R^m$, such that $\tilde{\parals_0}^{x} =\id_{\R^m}$ and the orthogonal transformation $\tilde{\parals_t}^{x} $ maps $\ker (X(x_0))$ to $\ker(X(x_t))$.  
We usually suppress the superscript $^x$ in the parallel translations and write simply $\parals_t$ and $\tilde{\parals_t}$.
\par
For $y\in M$, let $K(y)$ be the projection of $\R^m$ onto $\ker(X(y))$, and  
\[
K^{\perp}(y)=\id_{\R^m} - K(y)=Y_yX(y)
\] 
the projection onto the orthogonal complement of $\ker(X(y))$.  Then 
\[
\tilde{B}_t=\int_0^t \tilde{\parals_s}^{-1}K^{\perp}(x_s)dB_s
\quad\mbox{ and }\quad
\beta_t=\int_0^t\tilde{\parals_s}^{-1}K(x_s)dB_s.
\] 
\section{A preliminary Clark-Ocone formula for one-forms}
We write $-\div$ for the adjoint of the gradient operator $\nabla$.  It acts as a closed operator from $\Dom(\div)\subset G\otimes L^2(C_0;H)\simeq L^2(C_0;G\otimes H)$ to $L^2(C_0;G)$, for $G$ a separable Hilbert space. 
Given a function $g: C_0\to G$ in $\D^{2,1}$ and $V:C_0\to H$, we define $\nabla_Vg:C_{x_0}\to G$ by $\nabla_Vg(x)=dg_x(V(x))$. If $V$ is in the domain of $\div$, we have
 \[
\div (g\otimes V)=g\div V+\nabla _Vg.
\] 
We use the same notation when working on $C_{x_0}$.
\par
We identify $u\in\otimes^2 T_\sigma C_{x_0}$ with its evaluations $u_{s, t}\in T_{\sigma_s}M\otimes T_{\sigma_t}M$, continuous in $(s, t)$. 
Thus,  a vector $u\in\otimes^2 T_\sigma C_{x_0}$ is in $\otimes^2\HH$ if and only if we can write 
\[
u_{s, t} =(\WW_s\otimes \WW_t)(\frac{\D}{d.}\otimes \frac{\D}{d.})u,
\]
with $(\frac{\D}{d.}\otimes \frac{\D}{d.})u\in\otimes^2 L^2 T_\sigma C_{x_0}$. 
\par
Given two vector spaces $G$ and $K$, let $\tau:G\otimes K\to K\otimes G$ be the canonical flip map, i.e., $\tau(g\otimes k)=k\otimes g$, extended naturally to completed tensor products. 
 For inner product spaces $G$ and $ K$, let 
 $_{(1)}\!\langle-,h\rangle: G\otimes K\to K$, for $h\in G$,  
 be defined by 
 \[
_{(1)}\!\langle g\otimes k, h\rangle= \langle g,h\rangle_G k, \quad \forall g\in G, k\in K,
\]
and similarly $_{(2)}\!\langle-,l\rangle: G\otimes K\to G$, for $l\in K$,
by 
\[ 
_{(2)}\!\langle g\otimes k, \ell\rangle=g \langle k,l\rangle_K,  \quad \forall g\in G, k\in K. 
\]
Observe that $(\tau u)_{t,s} = \tau (u_{s,t})$ for $u\in\otimes^2 T_\sigma C_{x_0}$. 
\subsection{Differentiation of divergences}
\paragraph{The classical Wiener space.} 
The well-known commutation relationship between the derivative and divergence operators on the classical Wiener space is concisely expressed as 
$[\nabla, -\div]=\id_H$ in Nualart \cite{nualart2006malliavin}.  The following was given for abstract Wiener spaces by {\"U}st{\"u}nel and Zakai \cite{ustunel1995random} under slightly stronger conditions; see also \cite{yang2011thesis} for a proof.
\begin{lemma}[{\"U}st{\"u}nel and Zakai \cite{ustunel1995random}, Nualart \cite{nualart2006malliavin}]
\label{lem:UnZcommutation}
Given $U\in\D^{2,1}(C_0; H)$ and $\tau\nabla U\in\Dom(\div)$, we have $\div U \in \D^{2,1}$ and 
\[
\nabla(\div U)=\div\tau\nabla U -U. 
\]
In other words,  under these conditions,   we have $\div U \in \D^{2,1}$, and if $h\in H$, 
\begin{equation}
\label{eq:AWS_commutation}
\nabla_h(\div U)(x)=\langle (\div \tau\nabla U)(x), h\rangle_H - \langle   U(x),h\rangle_H.
\end{equation}
If $U\in\D^{2,2}(C_0; H)$ and $V\in\D^{2,1}(C_0; H)\cap L^\infty(C_0; H)$, we also have 
\begin{equation}
\label{eq:AWS-commutation-trace}
\nabla_V(\div U)
=\div\nabla_V U - \langle   U,V\rangle_H - \langle   \tau\nabla U,\nabla V\rangle_{H\otimes H}.
\end{equation}
\end{lemma}
When $U$ and $V$ are adapted, we only need to assume $U\in\D^{2,1}(C_0; H)$ and $V\in L^\infty(C_0;H)$ to obtain $\nabla_V U\in \Dom(\div)$, $\div U\in\D^{2,1}$, and 
\begin{equation}
\label{eq:AWS-commutation-adapted}
\nabla_V(\div U)=\div\nabla_V U - \langle   U,V\rangle_H.
\end{equation}
This is because the term $\langle \tau\nabla U,\nabla V\rangle_{H\otimes H}$ in  (\ref{eq:AWS-commutation-trace}) vanishes for adapted $U$ and $V$, reflecting the fact that adapted processes can be moved inside It\^o integrals. In particular, $\div\nabla_V U=(\div\nabla_-U)(V)$, where we treat  $\div\nabla_-U$ as a Hilbert-Schmidt operator valued integral.
\paragraph{The path space over $M$.} 
On the path space $C_{x_0}$, the following lemma extends slightly the commutation formula of Cruzeiro and Fang \cite{cruzeiro1997anl2estimate} (Theorem 3.2), with a different proof: while Cruzeiro and Fang  \cite{cruzeiro1997anl2estimate} used the stochastic development, we use the solution map of a suitable SDE on $M$.
\par
We let $\{\sigma_t\}_{0\le t\le T}$ denote both the canonical process on  $M$ and a generic element of $C_{x_0}$, and write the martingale part of the Stratonovich integral with respect to $\circ d\sigma_t$ as $d\{\sigma\}_t$.  
Thus $d\{\sigma\}_t= \parals_t d\breve{B}_t$, where $\breve{B}$ is the stochastic anti-development of our Brownian motion $\{\sigma_t\}_{t\in[0, T]}$ on $M$ using the given connection, and $d\{\sigma\}_t$ can be used to integrate suitable progressively measurable integrands.  
\par
Next, we recall a few useful results from \cite{elworthy2007ito}: 
\begin{eqnarray}
&\cdot& \mbox{Corollary 4.3}:
F\in \D^{2,1}(C_{x_0}; \R)\implies F\circ \I \in\D^{2,1}(C_0;\R); \label{fact1}\\
&\cdot& \mbox{Theorem 6.1}: 
F\circ \I\in \D^{2,1}(C_0;\R) \implies F\in \mathbb{W}^{2,1}(C_{x_0};\R);\label{fact2}\\
&\cdot& \mbox{Proposition 7.3}:
F\circ \I\in \D^{2,2}(C_0;\R) \implies F\in \D^{2,1}(C_{x_0};\R).\label{fact3}
\end{eqnarray}
Here the weak Sobolev space $\mathbb{W}^{2,1}$ is defined as the domain of the adjoint of the restriction of $d^*$ to $\D^{2,1}\HH^*$, so   
\[
 \mathbb{W}^{2,1}=\Dom((d^*|_{\D^{2,1}\HH^*})^*),
, \quad\mbox{ and } \D^{2,1}\subset \mathbb{W}^{2,1}.
\]
\begin{lemma}
\label{lem:commutation}
Suppose a vector field $U\in \D^{2,1}\HH$ is adapted to $\{\F_t^{{x_0}}\}_{t\in[0,T]}$. Then $\pnabla_{-} U$ is adapted and $\Div U\in\D^{2,1}$, with 
\begin{equation}
\label{eq:mnfd_commutation}
d(\Div U)(V) =-\int_0^T \langle   \frac{\D}{dt} (\,\pnabla_{-} U),  d\{\sigma\}_t\rangle_{\sigma_t} (V) - \langle  U, V\rangle_{\HH},
\end{equation}
for any $H$-vector field $V$ on $C_{x_0}$. 
If in addition $V$ is adapted, we have
\begin{equation}
\label{eq:mnfd_commutation_adapted}
d(\Div U)(V) = \Div(\,\pnabla_{V} U)\, - \langle   U, V\rangle_{\HH}.
\end{equation}
\end{lemma}
\begin{remark}
The map $\pnabla_{-} U\in \LL(\HH; \HH)$ is adapted in the sense that its composition with the evaluation map is adapted: 
$ev_t\circ\pnabla_{-} U\in \LL(\HH; T_{\sigma_t}M)$ is $\F_t^{{x_0}}$-measurable for all $t\in [0,T]$. 
The integral in (\ref{eq:mnfd_commutation}) is an $\LL(\HH;\R)$-valued It\^o integral.
\end{remark}
\begin{proof}
Choose an SDE on $M$ as in Section \ref{sec:Ito map} and use the notation of that section. 
Since $U$ is adapted to $\{\F_t^{{x_0}}\}_{t\in[0,T]}$, the map $\mathbb{Y} U\circ\I:C_0\to H$ is adapted to $\{\F_t^\I\}_{t\in[0,T]}$. 
We apply Corollary 5.2 in \cite{elworthy2007ito} to calculate 
\begin{eqnarray*}
(\Div U)\circ \I
&=& - \E\left[\int_0^T\langle  \frac{\D}{dt} U_t\circ \I, X(x_t) dB_t\rangle_{x_t}|\F_T^{\I}\right]\\
&=& - \int_0^T\langle  \frac{\D}{dt} U_t\circ \I, X(x_t) dB_t\rangle_{x_t}\\
&=& - \int_0^T\langle   Y\frac{\D}{dt} U_t\circ \I, K^\perp(x_t) dB_t\rangle_{\R^m},
\end{eqnarray*}
where the second line follows from the adaptedness of $U$.  
\par
Since $U\in \D^{2,1}(C_{x_0}; \HH)$, 
by (\ref{fact1}) we have 
$U\circ \I \in\D^{2,1}(C_0;\HH)$, hence $K^\perp(x_t)Y\frac{\D}{dt} U_t\circ \I \in\D^{2,1}(C_0;\R^m)$, so we can apply the commutation formula for the Wiener space to obtain 
$(\Div U)\circ \I\in \D^{2,1}(C_0;\R)$. 
Applying (\ref{fact2}), we see $\Div U\in \mathbb{W}^{2,1}(C_{x_0};\R)$. 
To prove $\Div U\in\D^{2,1}(C_{x_0};\R)$, we take a sequence of adapted 
$U_j\in \D^{2,2}\HH$ such that $U_j\to U$ in $\D^{2,1}\HH$.  
The argument above shows now  
$U_j\circ \I \in\D^{2,2}(C_0;\HH)$, 
$(\Div U_j)\circ \I\in \D^{2,2}(C_0;\R)$, and 
$
(\Div U_j)\circ \I\to (\Div U)\circ \I
$
in $\D^{2,1}(C_0;\R)$. 
Applying (\ref{fact3}), we also see that $(\Div U_j)\circ \I\in \D^{2,2}(C_0;\R)$ implies $\Div U_j\in \D^{2,1}(C_{x_0};\R)$. 
From Corollary 4.3 of \cite{elworthy2007ito} we know that the set
\[
\left\{f\circ\I|f\in\D^{2,1}(C_{x_0};\R) \right\}
\] 
is closed in $\D^{2,1}(C_0;\R)$, 
so the convergence of $(\Div U_j)\circ \I$ to $(\Div U)\circ \I$ in $\D^{2,1}(C_0;\R)$ implies that $\Div U \in \D^{2,1}(C_{x_0};\R)$.
\par
We now make use of the splitting (\ref{eq:splitting}) to calculate, for any $h\in H$, 
\begin{eqnarray}
\label{eq:DdivUIh}
& &d[(\Div U)\circ \I](h)\notag\\
&=& - \int_0^T\!\!\!\!\langle d[Y\frac{\D}{dt} U_t\circ \I](h),  K^\perp(x_t) dB_t\rangle_{\R^m}
- \int_0^T\!\!\!\!\langle Y\frac{\D}{dt} U_t\circ \I, K^\perp(x_t)\dot{h}_t\rangle_{\R^m}dt\notag\\
& &- \int_0^T\!\!\!\!\langle   Y\frac{\D}{dt} U_t\circ \I, d(K^\perp(x_t))(h)dB_t\rangle_{\R^m}\notag\\
&=&-\int_0^T\!\!\!\!\langle Xd(Y\frac{\D}{dt} U_t) T\I(-)_t, X(x_t)dB_t\rangle_{x_t}(h)
-\int_0^T\!\!\!\langle  \frac{\D}{dt} U_t\circ \I, X(x_t)\dot{h}_t\rangle_{\R^m}dt \notag\\
& &-\int_0^T\!\!\!\langle   \frac{\D}{dt} U_t\circ \I, Xd[YX]{T\I(-)_t}dB_t\rangle_{x_t}(h)\notag\\
&=&-\int_0^T\!\!\!\!\langle   \tilde{\nabla}_{T\I(-)_t}\frac{\D}{dt} U_t, X(x_t)dB_t\rangle_{x_t}(h)
-\int_0^T\!\!\!\langle  \frac{\D}{dt} U_t\circ \I, X(x_t)\dot{h}_t\rangle_{\R^m}dt \notag\\
& &-\int_0^T\!\!\!\langle \frac{\D}{dt} U_t\circ \I, \tilde{\nabla}_{T\I(-)_t}X(x_t)dB_t\rangle_{x_t}(h)\notag\\
&=&-\int_0^T\!\!\!\!\langle  \tilde{\nabla}_{T\I(-)_t}\frac{\D}{dt} U_t, X(x_t)\tilde{\parals_t} d\tilde{B}_t \rangle_{x_t}(h)
-\int_0^T\!\!\!\langle \frac{\D}{dt} U_t\circ \I, X(x_t)\dot{h}_t\rangle_{\R^m}dt \notag\\
& &-\int_0^T\!\!\langle   \frac{\D}{dt} U_t\circ\I, \tilde{\nabla}_{T\I(-)_t}X(x_t)\tilde{\parals_t}  d\beta_t\rangle_{x_t}(h),
\end{eqnarray}
where in the last line we used the basic property (\ref{eq:LJW_connection}). 
Recall the intertwining formula from \cite{elworthy2007ito} 
\begin{equation}
\label{eq:mnfd_intertwine}
df[\overline{T\I(h)}]=\overline{d(f\circ\I)(h)},
\end{equation}
where $f\in \D^{2,1}(C_{x_0};\R)$ and $h\in L^2(C_0; H)$. 
If $V=\overline{T \I}(h)$ for a constant $h\in H$, 
we can apply (\ref{eq:mnfd_intertwine}) to $f=\Div U$ and arrive at
\begin{equation}
\label{eq:DdivUV}
d(\Div U)(V(\sigma))=\E\left[d(\Div U\circ \I)(h)|\I=\sigma\right].
\end{equation}
To calculate this, first observe that taking conditional expectation of the right-hand side of (\ref{eq:DdivUIh}) 
annihilates the last term, since $\beta_{.}$ 
is independent of $\F^{x_0}_.$.  
Applying (\ref{eq:barTI}) and (\ref{eq:connection_relation}), we obtain from (\ref{eq:DdivUV})
\begin{eqnarray*}
&\!\! &d(\Div U)(V(\sigma)) \\
&\!\!=& -\int_0^T \langle \tilde{\nabla}_{ev_t(-)}\frac{\D}{dt}  U, d\{\sigma\}_t\rangle_{\sigma_t}(\overline{T\I}_\sigma(h)) -\int_0^T\langle  \frac{\D}{dt} U, X(\sigma_t)\dot{h}_t\rangle_{\R^m}dt\\
&\!\!=& -\int_0^T \langle   \tilde{\nabla}_{ev_t(-)}\frac{\D}{dt}  U, d\{\sigma\}_t\rangle_{\sigma_t}(V(\sigma)) -\int_0^T\langle  \frac{\D}{dt} U(\sigma),\frac{\D}{dt} V(\sigma)\rangle_{\R^m}dt\\
&\!\!=&-\int_0^T \langle   \frac{\D}{dt} (\,\pnabla_{-} U),  d\{\sigma\}_t\rangle_{\sigma_t} (V(\sigma))- \langle   U(\sigma), V(\sigma)\rangle_{\HH}.
\end{eqnarray*}
This finishes the proof of (\ref{eq:mnfd_commutation}) for a vector field of the form $V=\overline{T \I}(h)$ with $h\in H$. Since each term in (\ref{eq:mnfd_commutation}) is linear and continuous in $V$, we can extend this result immediately to a general $H$-vector field $V$.
\par
For $V$ an adapted $L^2$ $H$-vector field, equation (\ref{eq:mnfd_commutation_adapted}) holds since we can take $V$ inside the stochastic integral. 
\end{proof}
It would be  useful  to obtain a version of these formulae without the the adaptedness condition  on $U$. 
This would be a major step in obtaining a Weitzenb\"ock   identity for the Hodge Laplacian on  one-forms; see the works of Cruzeiro and Fang \cite{cruzeiro1997anl2estimate, cruzeiro2001weitzenbock,  cruzeiro2001weak} for related discussions.
\subsection{The operator $\mathfrak{D}^1$}
Define $\pnabla[\sharp]: \D^{2,1}\HH\to L^2\Gamma(\otimes^2 \HH)$ by 
\[\!\!{}_{(2)}\!\langle\,\pnabla[\sharp] U, V\rangle_{\HH}=\pnabla_V U, \quad \forall U\in \D^{2,1}\HH, V\in\HH.
\]
From equation (\ref{eq:DMconnection}), we see 
\begin{equation}
\label{eq:nabla^sharp_XY}
\pnabla[\sharp] U
= (\mathbb{X}\otimes\id)\nabla (\mathbb{Y} U).
\end{equation}
\par
We now define an operator $\mathfrak{D}^1:\Dom(\mathfrak{D}^1)\subset L^2\Gamma\HH\to L^2\Gamma(\wedge^2 \HH)$ by 
\[
\mathfrak{D}^1 V= \frac{1}{2}(\tau\pnabla[\sharp] V- \pnabla[\sharp] V),
\]
with initial domain $\D^{2,1}\HH$; we show below that it is closable and from then on take its closure.  
Let $\pnabla[*]:L^2\Gamma(\otimes^2 \HH)\to L^2\Gamma\HH$ 
be the $L^2$-adjoint of $\pnabla[\sharp]$. 
\par
In general for an element $v$ of a Hilbert space $G$, we denote by $v^\sharp\in G$ the dual element; for  a section $\phi$ of a Hilbert bundle,  $\phi^\sharp$ is the corresponding section of the dual bundle. Thus $(v^\sharp)^\sharp=v$, and if $\phi $ is an $H$-one-form, then $\phi^\sharp$ is an $H$-vector field.
\par
In terms of the operator $\mathfrak{D}^1$ and the torsion $\pathT\in \LL_{skew}(\HH,\HH; TC_{x_0})$ of the damped Markovian connection, formula (\ref{eq:pextder}) can now be written, for a smooth geometric one-form $\phi$ and two $H$-vector fields $V_1$, $V_2$, as
\begin{equation}
\label{eq:manifold_d1}
2 \, d^1\phi(V_1\wedge V_2)
= \langle  2 \, \mathfrak{D}^1\phi^{\sharp},   V_1\wedge V_2\rangle_{\wedge^2 \HH}  + (\phi\circ\pathT)(V^1,V^2).
\end{equation}
\begin{lemma}
\label{lem:mnfd_closability}
The operator $\mathfrak{D}^1$ is closable.  Moreover, 
\[
\mathfrak{D}^{1*}=-\pnabla[*]|_{L^2\Gamma(\wedge^2\HH)\cap \Dom(\,\pnabla[*])}.
\]
\end{lemma}
\begin{proof}
If $U\in L^2\Gamma(\wedge^2\HH)\cap\Dom(\,\pnabla[*])$, then $\tau U\!=\!-U\in L^2\Gamma(\wedge^2\HH)\cap\Dom(\pnabla[*])$. 
Thus $\pnabla[*]$ gives, when restricted to act on $\D^{2,1}(\wedge^2\HH)$, 
\[
\pnabla[*]U= - \frac{1}{2}(\tau\pnabla[\sharp] - \pnabla[\sharp])^*(U). 
\]
This implies that $U\in \Dom(\mathfrak{D}^{1*})$ and $\mathfrak{D}^{1*}U=-\pnabla[*]U$.
\par
The restriction of $\,\pnabla[*]$ to the intersection of its domain with $L^2\Gamma(\wedge^2 \HH)$ is a closed operator. 
It is densely defined in $L^2\Gamma(\wedge^2 \HH)$ by Proposition 9.6 of \cite{elworthy2008l2}.  
From this we see that the operator $\mathfrak{D}^1$ has the closed extension $\mathfrak{D}^{1**}$.
\par
It remains to show that $U\in \Dom(\mathfrak{D}^{1*})$ implies $U\in \Dom(\,\pnabla[*])$. For this, take $V\in \D^{2,1}\HH$ and $U\in \Dom(\mathfrak{D}^{1*})$. 
Since $U$ is skew-symmetric, 
\[
\E \langle\,\pnabla[\sharp] V,U\rangle 
=\frac{1}{2}\E \langle\,\pnabla[\sharp] V-\tau\pnabla[\sharp] V,U\rangle
=\E \langle V,- \frac{1}{2}(\tau\pnabla[\sharp]-\pnabla[\sharp])^*U\rangle,
\]
 and the result follows.
\end{proof}
\subsection{Differentiation through the optional projections} 
\label{se:diffproj}
Let  $P_{\V}:L^2\Gamma \HH\to L^2\Gamma \HH $ denote the projection onto the subspace $\V$ of adapted processes in $L^2\Gamma \HH$, i.e., 
\begin{equation}
\label{eq:mnfd_projection_adapted}
(P_{\V} U)_t = 
W_t\int_0^t W_s^{-1}\E(\frac{\D}{ds} U_s|\F^{x_0}_s)ds,\quad U\in L^2\Gamma \HH.
\end{equation} 
\par
We wish to show that $P_{\V}$ preserves the space of $\D^{2,1}$ vector fields.  This follows directly from Lemma \ref{lem:mnfd_conditional_D21} below, which extends the following result for the classical Wiener space by Nualart and Pardoux (Lemma 2.4 of \cite{nualart1988stochastic}): given $F\in \D^{2,1}(C_0; \R)$, we have $\E(F|\F_s)\in \D^{2,1}(C_0;\R)$ for $s\in [0,T]$, and 
\begin{equation}
\label{eq:NP}
\frac{d}{dt}[\nabla \E(F|\F_s)]_t 
= \E\left[\frac{d}{dt}(\nabla F)_t|\F_s\right]\mathbf{1}_{[0,s]}(t), \quad \mbox{a.e. 
 in } [0,T]\times C_0.
\end{equation}
\par
To simplify notation, we recall the definition  of the canonical resolution of the identity $\{\pi_s\}_{s\in [0,T]}$ on the Cameron-Martin space $H$ of the classical Wiener space; see \cite{ustunel1987construction} \cite{wu1990traitement}. 
This consists of projections $\pi_s:H\to H$, $s\in [0,T]$, given by
\begin{equation}
\label{eq:resolution}
(\pi_s h)_t=\int_0^{t\wedge s}\dot h_r dr, \quad \forall t\in [0,T], h\in H. 
\end{equation}
For brevity we write $s\lor t=\max(s,t)$, and $s\land t=\min(s,t)$.  We also recall  if $F\in \D^{2,1}(C_{x_0}; G)$, for $G$ a seperable Hilbert space,  then $\nabla F\in G\otimes L^2\Gamma\HH$. 
\begin{lemma}
\label{lem:mnfd_conditional_D21}
Given $G$ a seperable Hilbert space and $F\in \D^{2,1}(C_{x_0}; G)$, we have 
$\E(F |\F^{x_0}_s) \in \D^{2,1}(C_{x_0}; G)$ for $s\in [0,T]$, and a.e. in $[0,T]\times C_{x_0}$, 
\begin{equation}
\label{eq:mnfd_conditional_D21}
(\id \otimes \frac{\D}{dt})[\nabla \E(F|\F^{x_0}_{s})]_t 
=\mathbf{1}_{[0,s]}(t) \E\left[(\id \otimes\frac{\D}{dt})(\nabla F)_t|\F^{x_0}_{s}\right].
\end{equation}
\end{lemma}
\begin{proof}
We assume first $F\in \D^{2,2}(C_{x_0};\R)$, then 
$F\circ \I \in\D^{2,2}(C_0;\R)$ by (\ref{fact1}).  
Recall that for $F:C_{x_0}\to \R$, we have $F\circ \I\in\F_s$ if and only if $F\in\F_s^{{x_0}}$.  
Applying the Nualart-Pardoux result above, we see
\[
\E(F|\F^{x_0}_{s}) \circ \I = \E(F\circ \I |\F_s) \in \D^{2,2}(C_0; \R).
\]
Applying (\ref{fact3}) we see $\E(F|\F^{x_0}_{s})\in\D^{2,1}(C_{x_0};\R)$. 
Equation (\ref{eq:NP}) allows us to calculate, for $h\in H$, 
\begin{eqnarray*}
d[\E(F |\F^{x_0}_s)] \circ T\I(h)
&=& d [\E(F|\F^{x_0}_{s}) \circ \I] (h) \\
&=& d [\E(F\circ \I|\F_s)] (h)\\
&=& \E[d(F\circ \I )|\F_s](\pi_sh)\\
&=& \E[dF\circ T \I(\pi_sh)|\F_s], 
\end{eqnarray*}
so taking conditional expectation with respect to $\F^{x_0}_T$, we use (\ref{eq:2TI}) to obtain
\[
d[\E(F |\F^{x_0}_s)] \circ\overline{T \I}(h)
= \E[dF\circ \overline{T \I(\pi_sh)}|\F^{x_0}_{s}]
= \E[dF\circ \overline{T \I}(\pi_sh)|\F^{x_0}_{s}]. 
\]
This shows that, give any $h\in H$, 
\begin{eqnarray*}
& &\int_0^T\langle  \frac{\D}{dt}[\nabla \E(F|\F^{x_0}_{s})]_t, X(x_t) \dot h_t \rangle_{x_t}dt\\
&=&\E\left[\int_0^T\langle \frac{\D}{dt}\nabla F, X(x_t) \dot h_t \mathbf{1}_{[0,s]}(t)\rangle_{x_t}dt|\F^{x_0}_s\right]\\
&=&\int_0^T\langle  \E\left[\frac{\D}{dt}(\nabla F)_t|\F^{x_0}_{s}\right]\mathbf{1}_{[0,s]}(t), X(x_t) \dot h_t \rangle_{x_t}dt, 
\end{eqnarray*}
which proves the result for $F\in \D^{2,2}(C_{x_0};\R)$, since $X$ is onto. 
\par
For a general function $F\in \D^{2,1}(C_{x_0};\R)$, we take a sequence of functions $F_i\in \D^{2,2}(C_{x_0};\R)$ such that $F_i\to F$ in $\D^{2,1}(C_{x_0};\R)$. We now have 
$\E(F_i|\F^{x_0}_{s})\to\E(F|\F^{x_0}_{s})$ and  $\nabla F_i\to \nabla F$ in $L^2$.  
The above arguments also imply that $\E(F_i|\F^{x_0}_{s})\in \D^{2,1}(C_{x_0};\R)$, and 
\[
\frac{\D}{dt}[\nabla \E(F_i|\F^{x_0}_{s})]_t 
= \E[\frac{\D}{dt}(\nabla F_i)_t|\F^{x_0}_{s}]\mathbf{1}_{[0,s]}(t)
\to \E[\frac{\D}{dt}(\nabla F)_t|\F^{x_0}_{s}]\mathbf{1}_{[0,s]}(t)
\]
in $L^2$. 
Since $\nabla $ is a closed operator, we see 
$\E(F|\F^{x_0}_{s})\in \D^{2,1}(C_{x_0};\R)$, 
and (\ref{eq:mnfd_conditional_D21}) holds indeed for $F\in \D^{2,1}(C_{x_0};\R)$. 
\par
For $F\in \D^{2,1}(C_{x_0}; G)$, the result follows from the isomorphism between $\D^{2,1}(C_{x_0}; G)$ and $G\otimes\D^{2,1}(C_{x_0}; \R)$. 
\end{proof} 
\begin{proposition}
\label{prop:mnfd_NP}
Given $U\in \D^{2,1}\HH$,  we have $P_\V U\in \D^{2,1}\HH$, and 
\begin{equation}
\label{eq:mnfd_NP}
(\frac{\D}{ds}\otimes \frac{\D}{dt})(\pnabla[\sharp] P_\V U)_{s,t}
= \mathbf{1}_{[0,s]}(t) \E\left[(\frac{\D}{ds}\otimes \frac{\D}{dt})(\pnabla[\sharp]U)_{s,t}|\F^{x_0}_{s}\right].
\end{equation}
\end{proposition}
\begin{proof}
Since $U\in \D^{2,1}\HH$, we have $Y_{\sigma_s}\frac{\D}{ds}U_s\in \R^m$ for $\sigma \in C_{x_0}$ and $s\in [0,T]$. Lemma \ref{lem:mnfd_conditional_D21} 
shows that, a.e. in $[0,T]\times C_{x_0}$, 
\[
(\id\otimes\frac{\D}{dt})\left[\nabla\E(Y_{\sigma_s}\frac{\D}{ds}U_s |\F^{x_0}_s) \right]_t
=  \mathbf{1}_{[0,s]}(t)\E\left[(\id\otimes\frac{\D}{dt})\nabla (Y_{\sigma_s}\frac{\D}{ds}U)_t |\F^{x_0}_s\right].
\]
Equations (\ref{eq:nabla^sharp_XY}) and (\ref{eq:mnfd_projection_adapted}) allow us to derive 
\begin{eqnarray*}
(\frac{\D}{ds}\otimes \frac{\D}{dt})(\pnabla[\sharp] P_\V U)_{s,t}
&=&
(X(\sigma_s)\otimes \frac{\D}{dt})\nabla \left[Y_{\sigma_s}\E(\frac{\D}{ds}U_s| \F^{x_0}_s)\right]\\
&\!\!=\!\!& \mathbf{1}_{[0,s]}(t) \E\left[(\frac{\D}{ds}\otimes \frac{\D}{dt})(\pnabla[\sharp]U)_{s,t}|\F^{x_0}_{s}\right]. \qedhere
\end{eqnarray*}
\end{proof}
\subsection{The Clark-Ocone formula and its derivative}
The Clark-Ocone formula on $C_{x_0}$ was first obtained by S. Fang \cite{fang1994inegalite}; see also \cite{elworthy1999geometry}. It states that if $F\in \D^{2,1}(C_{x_0}; \R)$ then
\begin{equation}
\label{eq:mnfd_CO_0}
F(\sigma)=\E F + \int_0^T\langle  \E\left[\frac{\D}{dt}(\nabla F)_t|\F_t^{x_0}\right],d\{\sigma\}_t\rangle_{\sigma_t}, \quad \mu_{x_0}\mbox{-a.e. } \sigma\in C_{x_0}.
\end{equation}
In terms of the projection $P_{\V}$ discussed earlier,  formula (\ref{eq:mnfd_CO_0}) can be written as
\[
F=\E F - \Div P_{\V}\nabla F, \quad\forall F\in \D^{2,1}(C_{x_0}; \R). 
\]
\par
For an $L^2$ $H$-one-form $\phi$, we define $CO(\phi):C_{x_0}\to \R$ by
\[
CO(\phi)=-\Div P_{\V}\phi^\sharp.
\]
The Clark-Ocone formula shows immediately that, if an $L^2$ $H$-one-form $\phi$ is exact, it is the derivative of $CO(\phi)$. This motivates the following 
\begin{proposition}
\label{prop:mnfd_CO_1a}
If $\phi\in L^2\Gamma\HH^*$ is in $\Dom(\mathfrak{D}^1)$, 
we have $CO(\phi)\in \D^{2,1}$, and for a.e. $t\in[0,T]$, 
\begin{equation}
\label{eq:DCO}
\frac{\D}{dt}[\nabla  CO(\phi)- \phi^\sharp]_t
= 2 \int_t^T\!\!{}_{(2)}\!\langle  \E\left[(\frac{\D}{dt}\otimes \frac{\D}{ds})(\mathfrak{D}^1\phi^{\sharp})_{t,s}|\F_s^{x_0}\right],d\{\sigma\}_s\rangle_{\sigma_s}.
\end{equation}
Consequently, if $\phi^\sharp\in\Dom(\mathfrak{D}^1)$, then $\phi=df$ for some $f\in \D^{2,1}$ if and only if 
\[
\E\left[(\frac{\D}{dt}\otimes \frac{\D}{ds})(\mathfrak{D}^1\phi^{\sharp})_{t,s} |\F_{s\lor t}^{x_0}\right]=0, \quad \mbox{ a.e. } (s,t)\in [0,T]\times[0,T].
\]
\end{proposition}
\begin{proof}
Suppose first $\phi\in \D^{2,1}\HH^*$, so we can apply the Clark-Ocone formula (\ref{eq:mnfd_CO_0}) 
to write, for almost all $t\in [0,T]$,
\begin{equation*}
Y_{\sigma_t}\frac{\D}{dt}\phi^{\sharp}_t(\sigma) 
= \E Y_{\sigma_t}\frac{\D}{dt}\phi^{\sharp}_t + \int_0^T\!\!{}_{(2)}\!\langle  \E\left[\frac{\D}{ds}\nabla (Y_{\sigma_t}\frac{\D}{dt}\phi^{\sharp}_t)_s |\F_s^{{x_0}}\right]\!\!,d\{\sigma\}_s\rangle_{\sigma_s}.
\end{equation*}
Applying $X(\sigma_t)$ to both sides, we have
\[
\frac{\D}{dt}\phi^{\sharp}_t(\sigma) 
=X(\sigma_t)\E Y_{\sigma_t}\frac{\D}{dt}\phi^{\sharp}_t + X(\sigma_t)\int_0^T\!\!\!{}_{(2)}\!\langle  \E\left[\frac{\D}{ds}\nabla (Y_{\sigma_t}\frac{\D}{dt}\phi^{\sharp}_t)_s |\F_s^{{x_0}}\right]\!\!,d\{\sigma\}_s\rangle_{\sigma_s}.
\]
Taking conditional expectation with respect to $\F_t^{x_0}$, we obtain 
\begin{equation*}
\E(\frac{\D}{dt}\phi^{\sharp}_t
|\F_t^{x_0})
=X(\sigma_t)\E Y_{\sigma_t}\frac{\D}{dt}\phi^{\sharp}_t + X(\sigma_t)\int_0^t\!\!\!{}_{(2)}\!\langle  \E\left[\frac{\D}{ds}\nabla (Y_{\sigma_t}\frac{\D}{dt}\phi^{\sharp}_t)_s |\F_s^{{x_0}}\right]\!\!,d\{\sigma\}_s\rangle_{\sigma_s},
\end{equation*}
hence 
\begin{eqnarray}
\label{eq:manifold_phisharp_con_diff}
\frac{\D}{dt}\phi^{\sharp}_t(\sigma) - \E(\frac{\D}{dt}\phi^{\sharp}_t
|\F_t^{x_0})
&\!\!\!=\!\!\!& X(\sigma_t)\int_t^T\!\!\!{}_{(2)}\!\langle  \E\left[\frac{\D}{ds}\nabla (Y_{\sigma_t}\frac{\D}{dt}\phi^{\sharp}_t)_s |\F_s^{{x_0}}\right]\!\!,d\{\sigma\}_s\rangle_{\sigma_s}\notag\\
&\!\!\!=\!\!\!&\int_t^T\!\!\!{}_{(2)}\!\langle  \E\left[(\frac{\D}{dt}\otimes \frac{\D}{ds})(\pnabla[\sharp]\phi^{\sharp})_{t,s} |\F_s^{{x_0}}\right]\!\!,d\{\sigma\}_s\rangle_{\sigma_s}.
\end{eqnarray}
For $\phi\in \D^{2,1}\HH^*$, Proposition \ref{prop:mnfd_NP} shows $P_{\V}\phi^\sharp\in \D^{2,1}\HH$. 
By Lemma \ref{lem:commutation}, we have $CO(\phi)\in \D^{2,1}$, and for $V=\overline{T \I}(h)$ with $h \in H$, 
\begin{eqnarray*}
d\left[CO(\phi)\right](V)
&=&-d(\Div P_{\V}\phi^\sharp)(V)\\
&=&\int_0^T\langle  \frac{\D}{ds}(\pnabla_{-} P_{\V}\phi^\sharp), d\{\sigma\}_s\rangle_{\sigma_s}(V)
+\langle  P_{\V}\phi^\sharp, V\rangle_{\HH}.
\end{eqnarray*}
We can now use (\ref{eq:mnfd_NP}) and (\ref{eq:manifold_phisharp_con_diff}) to calculate
\begin{eqnarray*}
\label{eq:mnfd_phisharp_con_diff}
& &d\left[CO(\phi)\right](V) - \langle  \phi^\sharp, V\rangle_{\HH} \\
&=&\int_0^T\langle  \frac{\D}{ds}(\pnabla_{-} P_{\V}\phi^\sharp),d\{\sigma\}_s\rangle_{\sigma_s}(V)
-\langle  \phi^\sharp-P_{\V}\phi^\sharp, V\rangle_{\HH}\\
&=&\int_0^T\left\langle  \int_t^T\!\!{}_{(1)}\!\langle  \E\left[(\frac{\D}{ds}\otimes \frac{\D}{dt}) (\,\pnabla[\sharp] \phi^{\sharp})_{s,t}|\F_s^{x_0}\right],d\{\sigma\}_s\rangle_{\sigma_s},\frac{\D}{dt}V_t\right\rangle_{\sigma_t}dt\\
& & -\int_0^T\left\langle \int_t^T\!\!\!{}_{(2)}\!\langle  \E\left[(\frac{\D}{dt}\otimes \frac{\D}{ds})(\pnabla[\sharp]\phi^{\sharp})_{t,s} |\F_s^{{x_0}}\right],d\{\sigma\}_s\rangle_{\sigma_s}, \frac{\D}{dt}V_t\right\rangle_{\sigma_t}dt\\
&=&\!\!\!\!\int_0^T\left\langle\int_t^T\!\!{}_{(2)}\!\langle  \E\left[(\frac{\D}{dt}\otimes \frac{\D}{ds})(\tau\pnabla[\sharp] \phi^{\sharp}- \pnabla[\sharp]\phi^{\sharp})_{t,s} |\F_s^{x_0}\right],d\{\sigma\}_s\rangle_{\sigma_s},\frac{\D}{dt}V_t\right\rangle_{\sigma_t}dt\\
&=&\!\!\!\!\int_0^T\left\langle  2 \int_t^T\!\!{}_{(2)}\!\langle  \E\left[(\frac{\D}{dt}\otimes \frac{\D}{ds})(\mathfrak{D}^1\phi^{\sharp})_{t,s} |\F_s^{x_0}\right], d\{\sigma\}_s\rangle_{\sigma_s},\frac{\D}{dt}V_t\right\rangle_{\sigma_t}dt, 
\end{eqnarray*}
proving the results for $\phi\in \D^{2,1}\HH^*$. 
\par
Since $\nabla$ is a closed operator with domain $\D^{2,1}$, the result for general $\phi$ in $\Dom (\mathfrak{D}^1)$ follows from the continuity of the map 
$CO:L^2\Gamma \HH^*\to L^2(C_{x_0};\R)$ and the denseness of $\D^{2,1}\HH$ in $\Dom (\mathfrak{D}^1)$.
\end{proof}
In view of formula (\ref{eq:extder}) and the existence of torsion in general, it is not obvious how the preceding proposition would relate to a condition involving the vanishing of the exterior derivative.  This will, however, become clear once we have examined first the conditioning which appears in (\ref{eq:DCO}) and an extension of the concept of adaptedness to two-tensor fields, and then the definition of the exterior derivative and the space of two-vectors $\HH^{(2)}$.
\section{Adapted two-tensors}
The space of ``adapted" two-tensor fields which arises naturally in this context is the space $\V^{(2)}$ given by
\[
\V^{(2)}=\left\{U\in L^2\Gamma(\otimes^2\HH):
(\frac{\D}{ds}\otimes \frac{\D}{dt})U_{s,t}\in \F_{s\lor t}^{x_0}, \mbox{ a.e. } (s,t)\in [0,T]\times[0,T]\right\},
\] 
with the corresponding subspace $\V_{skew}^{(2)}$  of ``adapted"  $H$-two-vector fields
\[
\V_{skew}^{(2)}\!\!=\!\left\{U\in L^2\Gamma(\wedge^2\HH):
(\frac{\D}{ds}\otimes \frac{\D}{dt})U_{s,t}\in \F_{s\lor t}^{x_0}, \mbox{ a.e. } (s,t)\in [0,T]\times[0,T]\right\}.
\] 
Both $\V^{(2)}$ and $\V_{skew}^{(2)}$ are closed subspaces of $L^2\Gamma(\otimes^2\HH)$, and we define the orthogonal projection $P_{\V^{(2)}}$ onto $\V^{(2)}$ by 
\[
(\frac{\D}{ds}\otimes \frac{\D}{dt})P_{\V^{(2)}}(U)_{s,t}
= \E\left[(\frac{\D}{ds}\otimes \frac{\D}{dt})U_{s,t}|\F^{x_0}_{s\lor t}\right],  \mbox{ a.e. } (s,t)\in [0,T]\times[0,T].
\]
\par
Define 
the set $\mathcal{S}_b^{(2)} $ of \emph{adapted bounded primitive two-tensor fields} by 
\[
\mathcal{S}_b^{(2)}=\left\{V_1\otimes V_2: V_1, V_2 \in  L^\infty\Gamma\HH \mbox{, both   adapted to }\{\F^{x_0}_t\}_{t\in[0,T]}\right\},
\]
as well as the corresponding set 
$\mathcal{S}_{b,skew}^{(2)}$ of \emph{adapted primitive two-vector fields},  with $V_1\wedge V_2$ replacing $V_1\otimes V_2$ in the definition.
\begin{lemma}
\label{lem:manifold_totality}
$\mathcal{S}_b^{(2)}$ is total in $\V^{(2)}$.  Similarly, $\mathcal{S}_{b,skew}^{(2)}$  is total in $\V_{skew}^{(2)}$.
\end{lemma}
\begin{remark}
Equivalently, we could state the $L^2$ version of the lemma, replacing 
$\V^{(2)}$ and $\mathcal{S}_b^{(2)} $, respectively, with
\[
{\V^{(2)}}'=\left\{U\in L^2\Gamma(\otimes^2 L^2 T C_{x_0}):
U_{s,t}\in \F_{s\lor t}^{x_0}, \mbox{ a.e. } (s,t)\in [0,T]\times[0,T]\right\},
\] 
and 
\[
{\mathcal{S}_b^{(2)}}'=\left\{V_1\otimes V_2: V_1, V_2 \in L^\infty\Gamma (L^2 T C_{x_0}), \mbox{ both adapted to }\{\F^{x_0}_t\}_{t\in[0,T]}\right\}.
\]
\end{remark}
\begin{proof}
Take an orthonormal basis $\{h^{x_0,j}\}_{j\in\mathbb{N}}$ of $L^2([0,T]; T_{x_0}M)$, and set 
\[
h^j_t(\sigma)=\parals_t^{\sigma}h^{x_0,j}_t,\quad j\in\mathbb{N}, t\in [0,T], \sigma\in C_{x_0}. 
\]
Given any $U\in\V^{(2)}$, we can approximate 
$(\frac{\D}{ds}\otimes \frac{\D}{dt})U_{s,t}\in \F_{s\lor t}^{x_0}$ by finite sums 
\[
\sum_{j,k=1}^n\lambda^{j,k}_{s\lor t} \, h^j_s\otimes h^k_t,
\] 
where $\lambda^{j,k}:[0,T]\times C_{x_0} \to \R$ are bounded and adapted, i.e., $\lambda^{j,k}_{r}\in \F_{r}^{x_0}$ for all $r\in[0,T]$.
Therefore, we only need to show that each term in the finite sums above can be approximated in $L^2$ by sums of terms of the form $(V_1)_s\otimes (V_2)_t$, 
with $V_1$, $V_2\in L^\infty\Gamma(L^2TC_{x_0})$, both adapted.  
\par
Since
$\lambda^{j,k}_{r} $ 
is the $L^2$ limit of sums of bounded elementary processes of the form 
\[
(r, \sigma)\mapsto f(\sigma)\mathbf{1}_{(a, b]}(r), \quad  \sigma\in C_{x_0}, a, b\in[0,T], f\in \F_{a}^{x_0},
\]
we write $\lambda^{j,k}_{s\lor t}$ as the limit of finite sums of functions of form $f(\sigma)\mathbf{1}_{(a,b]}(s\lor t)$. 
Observe that
\begin{eqnarray*}
\mathbf{1}_{(a,b]}(s\lor t) 
&=& \mathbf{1}_{[0,a]}(s)\mathbf{1}_{(a,b]}(t) + \mathbf{1}_{[0,a]}(t)\mathbf{1}_{(a,b]}(s) + \mathbf{1}_{(a,b]}(t)\mathbf{1}_{(a,b]}(s)\\
&=& \mathbf{1}_{[0,b]}(s)\mathbf{1}_{(a,b]}(t) + \mathbf{1}_{[0,a]}(t)\mathbf{1}_{(a,b]}(s),
\end{eqnarray*}
which means
\begin{eqnarray*}
& &f(\sigma)\mathbf{1}_{(a,b]}(s\lor t)h^j_s\otimes h^k_t \\
&=& \mathbf{1}_{[0,b]}(s)h^j_s\otimes f(\sigma)\mathbf{1}_{(a,b]}(t)h^k_t 
+ f(\sigma) \mathbf{1}_{(a,b]}(s)h^j_s\otimes \mathbf{1}_{[0,a]}(t) h^k_t. 
\end{eqnarray*}
This shows that $\lambda^{j,k}_{s\lor t} h^j_s\otimes h^k_t$ is indeed a limit of sums of $(V_1)_s\otimes (V_2)_t$ with $V_1$ and $V_2$ adapted, so the conclusion holds. The result for $\V_{skew}^{(2)}$ follows by skew-symmetrisation.
\end{proof}
\subsection{The subspaces $\V^{(2)}_1$ and $\V^{(2)}_2$}
We decompose $P_{\V^{(2)}}$ to write $P_{\V^{(2)}}= P_{\V^{(2)}_1}+ P_{\V^{(2)}_2}$, where the two components are given respectively by, for $U\in L^2\Gamma(\otimes^2\HH)$ and  a.e. $(s,t)\in [0,T]\times[0,T]$,
\[
(\frac{\D}{ds}\otimes \frac{\D}{dt})P_{\V^{(2)}_1}(U)_{s,t}
= \mathbf{1}_{[0,t)}(s) \E\left[(\frac{\D}{ds}\otimes \frac{\D}{dt})U_{s,t}|\F^{x_0}_{s\lor t}\right],  
\]
and
\[
(\frac{\D}{ds}\otimes \frac{\D}{dt})P_{\V^{(2)}_2}(U)_{s,t}=  \mathbf{1}_{[0,s)}(t) \E\left[(\frac{\D}{ds}\otimes \frac{\D}{dt})U_{s,t}|\F^{x_0}_{s\lor t}\right], 
\]
with images denoted by $\V^{(2)}_1$ and $\V^{(2)}_2$, respectively.  
It is easy to see that
\begin{eqnarray*}
\left[(\frac{\D}{ds}\otimes \frac{\D}{dt}) \tau P_{\V^{(2)}_1}(U)\right]_{s,t} 
&=& 
\tau \left[(\frac{\D}{dt}\otimes \frac{\D}{ds}) P_{\V^{(2)}_1}(U)\right]_{t,s} \\
&=& \tau\left(\mathbf{1}_{[0,s)}(t) \E\left[(\frac{\D}{dt}\otimes \frac{\D}{ds})U_{t,s}|\F^{x_0}_{s\lor t}\right]\right) \\
&=& \mathbf{1}_{[0,s)}(t) \E\left[(\frac{\D}{ds}\otimes \frac{\D}{dt})(\tau U)_{s,t}|\F^{x_0}_{s\lor t}\right],
\end{eqnarray*}
so the two components are related by the following identity
\begin{equation}
\label{eq:skew-symmetrisation}
\tau P_{\V^{(2)}_1} = P_{\V^{(2)}_2}\tau.
\end{equation}
Equation (\ref{eq:mnfd_NP}) in Proposition 3.6 can be expressed as
\[
\pnabla[\sharp] P_\V = P_{\V^{(2)}_2}\pnabla[\sharp],
\]
which implies that the covariant derivative of any adapted $\D^{2,1}$ $H$-vector field lies inside $\V^{(2)}_2$.  That $ \pnabla[\sharp] [\V] \subset \V^{(2)}_2$ may also be compared to the fact that $\,\pnabla[*] [\V^{(2)}_1 ] \perp \V$, proved below in Lemma \ref{lem:nabla*}. 
Similar observations regarding higher order vector fields on the Wiener space were mentioned in  \cite{yang2011generalised}.
\subsection{The domain and range of $\pnabla[*]$}
Given any $U\in L^2\Gamma(\otimes^2\HH)$ and $V_1\in \D^{2,1}\HH$, we see from identity (\ref{eq:nabla^sharp_XY}) that 
\[
\E \langle\,\pnabla[\sharp] V_1, U\rangle_{\otimes^2\HH}  
= \E \langle (\mathbb{X}\otimes \id)\nabla(\mathbb{Y}V_1), U\rangle_{\otimes^2\HH} 
= - \E \langle V_1, \mathbb{X}\, \div[(\mathbb{Y}\otimes \id)U]\rangle_{\HH}.
 \]
So in general, we have $U\in\Dom(\,\pnabla[*])$ and 
$\pnabla[*]U=-\mathbb{X}\div[(\mathbb{Y}\otimes \id)U]$ provided 
$(\mathbb{Y}\otimes\id)U\in\Dom(\div)$.  
For the special case of $U\in\V_1^{(2)} \subset L^2\Gamma(\otimes^2\HH)$, we have the following 
\begin{lemma}
\label{lem:nabla*}
Suppose $U\in\V^{(2)}_1$. Then 
$U \in\Dom(\pnabla[*])$ and 
\begin{equation}
\label{eq:pnabla*}
\frac{\D}{dt}(\,\pnabla[*]U) =
\int_t^T \!\!{}_{(2)}\!\langle (\frac{\D}{dt}\otimes \frac{\D}{ds})U_{t,s},  
d\{\sigma\}_s\rangle_{\sigma_s}.
\end{equation}
Thus, $\V^{(2)}_1 \subset\Dom(\,\pnabla[*])$ and $\,\pnabla[*] [\V^{(2)}_1 ] \perp \V$.
\end{lemma}
\begin{proof}
Following the preceding remark, we verify $(\mathbb{Y}\otimes\id)U\in\Dom(\div)$. 
By the definition of $\mathbb{Y}$ in (\ref{eq:bold_Y}), we have, for a.e. $s\in[0,T]$, 
\begin{eqnarray*}
(\id\otimes\frac{\D}{ds})[(\mathbb{Y}\otimes\id)U]_{t,s}
&=&(\id\otimes\frac{\D}{ds})\int_0^t(Y_{\sigma_r}\frac{\D}{dr}\otimes\id)U_{r,s}dr\\
&=&\int_0^{s\land t}(Y_{\sigma_r}\frac{\D}{dr}\otimes\frac{\D}{ds})U_{r,s} dr,
\end{eqnarray*}
which is $\F_{s}^{x_0}$-measurable for all $s$. 
Hence indeed $(\mathbb{Y}\otimes\id)U\in\Dom(\div)$, and $\div[(\mathbb{Y}\otimes\id)U]$ is the It\^o-integral given by
\begin{eqnarray*}
\div[(\mathbb{Y}\otimes\id)U]_t
&=& -\int_0^T\!\!{}_{(2)}\!\langle\int_0^{s\land t}(Y_{\sigma_r}\frac{\D}{dr}\otimes\frac{\D}{ds})U_{r,s} dr,d\{\sigma\}_s\rangle_{\sigma_s}\\
&=& -\int_0^t\int_r^T\!\!{}_{(2)}\!\langle(Y_{\sigma_r}\frac{\D}{dr}\otimes\frac{\D}{ds})U_{r,s},d\{\sigma\}_s\rangle_{\sigma_s} dr,
\end{eqnarray*}
where we used the fact that  $(Y_{\sigma_r}\frac{\D}{dr}\otimes\frac{\D}{ds})U_{r,s} \in \F_s^{x_0}$ to apply Fubini's theorem.  
We conclude that 
\begin{eqnarray*}
\frac{\D}{dt}(\,\pnabla[*]U) 
&=& - \frac{\D}{dt}\mathbb{X}\,\div(\mathbb{Y}\otimes \id)U\\
&=& X_{\sigma_t}\frac{d}{dt}\left[\int_0^t\int_r^T\!\!{}_{(2)}\!\langle(Y_{\sigma_r}\frac{\D}{dr}\otimes\frac{\D}{ds})U_{r,s},d\{\sigma\}_s\rangle_{\sigma_s} dr\right]\\
&=& \int_t^T\!\!{}_{(2)}\!\langle(\frac{\D}{dt}\otimes\frac{\D}{ds})U_{t,s},d\{\sigma\}_s\rangle_{\sigma_s}.\qedhere
\end{eqnarray*}
\end{proof}
\begin{remark}
\label{re:isometry_subsp}
From the lemma, we see that $\V^{(2)}_1$ is an isometry subspace of $L^2\Gamma(\otimes^2\HH)$ in the sense of Wu \cite{wu1990traitement} for the map $\,\pnabla[*]:L^2\Gamma(\otimes^2\HH)\to L^2\Gamma\HH$, i.e., 
\[ 
\|\,\pnabla[*]U \|_{L^2\Gamma\HH} = \|U \|_{L^2\Gamma(\otimes^2\HH)},\quad \forall U\in\V^{(2)}_1.
 \]
Indeed, applying the It\^o isometry we see 
\begin{eqnarray*}
\|\,\pnabla[*]U \|^2_{L^2\Gamma\HH}
&=&\E\int_0^T|\frac{\D}{dt}(\pnabla[*]U)_t|^2 dt\\
&=&\E\int_0^T|\int_t^T \!\!{}_{(2)}\!\langle (\frac{\D}{dt}\otimes \frac{\D}{ds})U_{t,s}, d\{\sigma\}_s\rangle_{\sigma_s}|^2 dt\\
&=&\int_0^T\int_t^T \E |(\frac{\D}{dt}\otimes \frac{\D}{ds})U_{t,s}|^2 dsdt\\
&=&\|U \|^2_{L^2\Gamma(\otimes^2\HH)}.
\end{eqnarray*}
\end{remark}
\section {The space of two-vectors $\HH^{(2)}\!$ and the exterior derivative $d^1$}
We define $\HH^{(2)}$ as a perturbation of $\wedge^2\HH$ in $\wedge^2 TC_{x_0}$, i.e., for $u\in \wedge^2 TC_{x_0}$, 
 \begin{eqnarray*}
u\in {\HH}^{(2)}
&\Longleftrightarrow
&u-\pathR(u)\in \wedge^2 {\cal H}\\
&\Longleftrightarrow& u=(\id+Q)(v),\quad \text{for some } v\in \wedge^2\cal{H}; 
\end{eqnarray*}
see \cite{elworthy2000special, elworthy2008l2}. Here $\pathR:\wedge^2TC_{x_0}\to \wedge^2 TC_{x_0}$ is the curvature operator of the damped Markovian connection, and the linear map $Q_\sigma:\wedge^2 \HH\to \wedge^2T_\sigma C_{x_0}$  
is defined by 
\[
Q_\sigma(G)_{s,t} = (W_s\otimes W_t) \left(\wedge^2 (W_.^{-1}) W_.^2\int_0^. (W_r^2)^{-1} [\RR_{\sigma_r}(G_{r,r})]dr\right)_{s\land t},
\]
with $W^2_t:\wedge^2T_{x_0}M\to \wedge^2T_{\sigma_t}M$ being the damped parallel translation of two-vectors using the second Weitzenb\"ock curvature, and $\RR:\wedge^2TM\to\wedge^2TM$ the curvature operator on $M$. 
\par
In particular,  $1 + Q$ and $1-\pathR$ are inverse of each other (\cite{elworthy2008l2} Lemma 4.2). By requiring that these operators give an isometry between $\HH^{(2)}$ and $\wedge^2\HH$, we can define a Riemannian structure on $\HH^{(2)}$ using which, almost surely, 
\begin{equation}
\label{eq:duality}
1-\pathR=(1+Q)^*: \HH^{(2)}\to \wedge^2\HH.
\end{equation}
As in \cite{elworthy2000special, elworthy2008l2}, an exterior differentiation  operator 
\[
d^1\!\!: \Dom(d^1)\subset L^2\Gamma\HH^*\to L^2\Gamma (\HH^{(2)*})
\] 
can be defined as the closure of the geometric exterior derivative acting on cylindrical one-forms and then restricted to the fibres of  $\HH^{(2)}$. That is, for a smooth cylindrical one-form $\phi$, our new exterior derivative
of $\phi$ restricted to $\HH$ is just $d^1\phi|_{\HH^{(2)}}$, 
where $d^1$ refers to the geometric exterior derivative. We then have $d^1d=0$ and a Hodge decomposition
\begin{equation}
\label{eq:path_Hodge1a}
L^2\Gamma \HH^*= {\image(d)}\bigoplus \mathbf{H}^1(C_{x_0})\bigoplus \overline{\image(d^{1*})}.
\end{equation}
\subsection{$\Div^1$ and $\DIV^1$} 
By definition, the divergence operator $\Div^1:\Dom (\Div^1)\subset L^2\Gamma\HH^{(2)}\to L^2\Gamma\HH$ is minus the co-joint of $d^1$, i.e., 
\[
\int_{C_{x_0}} d^1\phi(U)d\mu=-\int_{C_{x_0}}\phi( \Div^1U)d\mu, \quad \phi\in \Dom(d^1),U\in \Dom(\Div^1).
\]
Its domain is the set $\{U\in L^2\Gamma\HH^{(2)}: U^\sharp\in \Dom(d^{1*})\}$, on which we have $(\Div^1 U)^\sharp=- d^{1*}(U^\sharp)$ analogously to the usual divergence acting on $H$-vector fields.
\par
More generally, a measurable geometric two-vector field $U\in \Gamma(\wedge^2TC_{x_0})$ is said to \emph{have a divergence} if there exists an integrable vector field $\DIV^1U\in \Gamma \HH$ such that
$d^1\phi(U):C_{x_0}\to \R$ is integrable  for all smooth  cylindrical one-forms $\phi$ on $C_{x_0}$, and
\[
\int_{C_{x_0}} d^1\phi(U)d\mu=-\int_{C_{x_0}}\phi(\DIV^1U)d\mu.
\]
A class of examples of such vector fields on the classical Wiener space was given in Section 8 of \cite{elworthy2008l2}. 
Similarly, the operator $\DIV$ extends $\div$ for geometric vector fields.  
\par
Note that if $U$ takes values in $\HH^{(2)}$ and is in the domain of $\Div^1$, then it has a divergence in this extended sense and $\DIV^1 U=\Div^1 U$. 
\par
The following result from Theorem 9.3 of \cite{elworthy2008l2} appears to be crucial in our theory. Earlier Cruzeiro and Fang \cite{cruzeiro1997anl2estimate} had shown that, for a class of adapted primitive $U$, the geometric vector field $\pathT(U)$ has a vanishing divergence.
\begin{theorem}[\cite{elworthy2008l2}]
\label{th:manifold_torsion_div}
Suppose $U\in \mathcal{S}_{b,skew}^{(2)}$. Then the geometric two-vector field $Q(U)$ has a divergence and
\[
\DIV^1[Q(U)]=\frac{1}{2} \pathT(U).
\]
Therefore, $\DIV(\pathT(U))=0$, i.e.,  given any smooth cylindrical $f:C_{x_0}\to \R$, 
\[
\E\, df[\pathT(U)] = 0.
\]
\end{theorem}
\begin{proposition}
If $U\in \V_{skew}^{(2)}$ and $\phi$ is a smooth cylindrical one-form, 
\begin{equation}
\label{eq:key}
\E\, d^1\phi[(1+Q)U]=\E \langle  \mathfrak{D}^1\phi^\sharp, U\rangle_{\wedge^2\HH}.
\end{equation}
\end{proposition}
\begin {proof}
We can suppose $U\in \mathcal{S}_{b,skew}^{(2)}$; the general result follows by approximation using Lemma \ref{lem:manifold_totality}. Application of  (\ref{eq:manifold_d1}) and Theorem \ref{th:manifold_torsion_div} then yields 
\begin{eqnarray*}
\E\,d^1\phi\left(U\right)
&=&\E\langle \mathfrak{D}^1\phi^\sharp, U\rangle_{\wedge^2\HH} +
\frac{1}{2}\E\,\phi[\pathT(U)] \\
&=& \E\langle  \mathfrak{D}^1\phi^\sharp, U\rangle_{\wedge^2\HH} - \E\,d^1\phi[Q(U)].
\qedhere
\end{eqnarray*}
\end{proof}
\subsection{$d^1$ and $\overline{\mathfrak{D}^1}$}
It is convenient to define a new operator $\overline{\mathfrak{D}^1}$ from  $L^2\Gamma \HH^*$ to $L^2\Gamma (\wedge^2 \HH)$ with initial domain 
$\{\phi\in L^2\Gamma \HH^* :\phi^\sharp\in\Dom(\mathfrak{D}^1)\}$, by 
\begin{equation}
\label{eq:manifold_barD1_D1}
\overline{\mathfrak{D}^1} \phi= P_{\V^{(2)}}\mathfrak{D}^1\phi^{\sharp}.
\end{equation}
Recall that $P_{\V^{(2)}}$ is the projection of $L^2\Gamma (\otimes^2 \HH)$ onto $\V^{(2)}$.  Thus
\[
(\frac{\D}{ds}\otimes \frac{\D}{dt})(\overline{\mathfrak{D}^1} \phi)_{s,t} =
\E[(\frac{\D}{ds}\otimes \frac{\D}{dt})(\mathfrak{D}^1\phi^{\sharp})_{s,t} |\F_{s\lor t}^{x_0}].
\]
Using the following lemma, we take the closure of $\overline{\mathfrak{D}^1}$ without change of notation. 
\begin{lemma}
\label{le:manifold_barD_closable}
The operator $\overline{\mathfrak{D}^1}$ is closable on $\D^{2,1}$. 
\end{lemma}
\begin{proof}
Given any sequence of $\phi_j\in \D^{2,1}$, 
$j\in \mathbb{N}$, such that $\phi_j\to 0$ and $\overline{\mathfrak{D}^1}\phi_j\to U$ in $L^2$, 
Proposition \ref{prop:mnfd_CO_1a} implies 
\begin{eqnarray*}
\frac{\D}{dt}[\nabla  CO(\phi_j)]_t 
&=& 
\frac{\D}{dt}(\phi_j^\sharp)_t +2\int_t^T\!\!{}_{(2)}\!\langle  (\frac{\D}{dt}\otimes \frac{\D}{ds})(\overline{\mathfrak{D}^1}\phi_j)_{t,s},d\{\sigma\}_s\rangle_{\sigma_s}\\
&\to&
2 \int_t^T\!\!{}_{(2)}\!\langle  (\frac{\D}{dt}\otimes \frac{\D}{ds})U_{t,s},d\{\sigma\}_s\rangle_{\sigma_s}
\end{eqnarray*} 
in $L^2$. 
Note that $U\in P_{\V^{(2)}}$, so $U_{t,\cdot}$ is adapted on $(t,T]$.  
Since $CO(\phi_j) \to 0$ and $\nabla  $ is closed, we get $\frac{\D}{dt}[\nabla  CO(\phi_j)]_t \to 0$.  Hence the It\^o integral
\[
2 \int_t^T\!\!{}_{(2)}\!\langle  (\frac{\D}{dt}\otimes \frac{\D}{ds})U_{t,s},d\{\sigma\}_s\rangle_{\sigma_s}=0, \quad \mbox{ a.e. }t\in [0,T], 
\]
so $U=0$, i.e., $\overline{\mathfrak{D}^1}\phi_j\to 0$.  
\end{proof}
The next result explains the relationship between $d^1$ and $\mathfrak{D}^1$, which we have been working towards. A more definitive version
is given in Corollary \ref{co:manifold_don_e_closable} of the next section.  Recall that by $(d^1\phi)^\sharp$, for $\phi\in \Dom( d^1)$ and in particular for a smooth cylindrical one-form, we mean the element
of $L^2\Gamma \HH^{(2)}$ corresponding to $d^1\phi\in L^2\Gamma(\HH^{(2)*})$ using the Riemannian structure of $\HH^{(2)}$.
\begin{lemma}
\label{le:sought}
For all smooth cylindrical one-forms $\phi$, 
\begin{equation}
\label{eq:manifold_barD1_d1}
\overline{\mathfrak{D}^1}\phi = P_{\V^{(2)}} (1-\pathR)
(d^1\phi)^\sharp.
\end{equation}
\end{lemma}
\begin{proof}
 Take any $U\in L^2\Gamma(\wedge^2\HH)$. By equations (\ref{eq:key}) and (\ref{eq:duality}), 
\[
\E \langle\overline{\mathfrak{D}^1}\phi, U\rangle_{\wedge^2\HH}
=\E \,d^1\phi[(1+Q)P_{\V^{(2)}} U]
=\E\langle  P_{\V^{(2)}} (1-\pathR)(d^1\phi)^\sharp, U\rangle_{\wedge^2\HH}.\qedhere
\]
\end{proof}
\section{A Clark-Ocone formula for one-forms and cohomology vanishing}
We now present a Clark-Ocone type formula for one-forms.  A  variety of such formulae including formulae for higher order forms were given on Wiener spaces  by Yang \cite{yang2011thesis, yang2011generalised}.
\begin{theorem}
\label{th:mnfd_CO_1}
If $\phi\in L^2\Gamma\HH^*$ is in $\Dom(d^1)$, 
we have $CO(\phi)\in \D^{2,1}$ and
\begin{equation}
\label{eq:CO1}
 \frac{\D}{dt}[\nabla  CO(\phi)- \phi^\sharp]_t
= 2\int_t^T\!\!{}_{(2)}\!\langle (\frac{\D}{dt}\otimes \frac{\D}{ds})[P_{\V^{(2)}} (1-\pathR)(d^1\phi)^\sharp], d\{\sigma\}_s\rangle_{\sigma_s}.
\end{equation}
\end{theorem}
\begin{proof}
By Lemma \ref {le:sought} and 
Proposition \ref{prop:mnfd_CO_1a}, the theorem holds for smooth cylindrical $\phi$. Since such $\phi$ are dense in the domain of $d^1$, we argue as in the proof of Lemma  \ref{le:manifold_barD_closable} to obtain the result for all $\phi\in\Dom(d^1)$.
\end{proof}
Lemma \ref{lem:nabla*} enables us to write (\ref{eq:CO1}) in a concise form.
\begin{corollary}
\label{co:manifold_CO_short}
Any $\phi\in \Dom(d^1)$ can be expressed as 
\begin{equation}
\label{eq:CO2}
\phi^\sharp
= \nabla CO(\phi) - 2\, \pnabla[*][P_{\V^{(2)}_1} (1-\pathR)(d^1\phi)^\sharp].
\end{equation}
\end{corollary}
\begin{corollary}
\label{co:manifold_don_e_closable}
With initial domain that of $d^1$, the operator from $L^2 \Gamma\HH^*$ to $L^2\Gamma(\wedge^2\HH)$ given by 
\[
\phi\mapsto P_{\V^{(2)}} \left(1-\pathR\right)(d^1\phi)^\sharp
\]
is closable. Taking its closure we have the following equality of closed densely defined operators
\begin{equation}
\label{eq:equal_ops}
\overline{\mathfrak{D}^1}=\left[\phi\mapsto P_{\V^{(2)}} \left(1-\pathR\right)(d^1\phi)^\sharp\right].
\end{equation}
\end{corollary}
\begin{proof}
For the required closability, we apply the argument in the proof of Lemma \ref{le:manifold_barD_closable} to equation (\ref{eq:CO1}). We then use equation (\ref {eq:manifold_barD1_d1}) and the fact that smooth cylindrical forms are dense in the domains of both $d^1$ and $\overline{\mathfrak{D}^1}$.
\end{proof}
\begin{corollary}
\label{cor:mnfd_CO_bound}
If $\phi\in L^2\Gamma\HH^*$ is in $\Dom(d^1)$, we have 
\[
\|dCO(\phi)- \phi\|_{L^2\Gamma\HH^*} 
=\sqrt{2} \|\overline{\mathfrak{D}^1}\phi\|_{L^2\Gamma(\otimes^2 \HH)} 
\le \sqrt{2}\|d^1\phi\|_{L^2\Gamma(\HH^{(2)*})}. 
\]
In particular, if $\phi$ is also orthogonal to the image of $d$, then
\[
\|\phi\|_{L^2\Gamma\HH^*} 
\le \sqrt{2}\|d^1\phi\|_{L^2\Gamma(\HH^{(2)*})}.
\]
\end{corollary}
\begin {proof}
The results are clear from (\ref{eq:CO1}) and (\ref{eq:manifold_barD1_d1}), using the It\^o isometry as in Remark \ref{re:isometry_subsp} and the fact that 
$1-\pathR$ gives an isometry from $\HH^{(2)}$ to $\wedge^2\HH$.
\end{proof}
We also obtain a Clark-Ocone type formula for co-closed one-forms.  
Recall first the projection operators $P_{\V^{(2)}_1}$ and $P_{\V^{(2)}_2}$, which are related by (\ref{eq:skew-symmetrisation}). 
\begin{corollary}
\label{cor:mnfd_co-closed}
If $\phi\in L^2\Gamma\HH^*$ is in  $\ker(d^*)$, we have 
$\phi = d^{1*}w$, where 
\[
w^\sharp=(1+Q)( P_{\V^{(2)}_2}\tau - P_{\V^{(2)}_1})
\pnabla[\sharp]\phi^{\sharp}. 
\]
\end{corollary}
\begin {proof}
Since $\phi\in \ker(d^*)$ means $\phi^\sharp\in \ker(\div)$, we see for any $\psi\in\Dom(d^1)$ that
\[
0 = \E [-(\div \phi^\sharp)CO(\psi)] = \E \langle\phi^\sharp, \nabla CO(\psi)\rangle_\HH.
\]
Applying Corollary \ref{co:manifold_CO_short} to $\psi\in\Dom(d^1)$, we see  
\begin{eqnarray*}
\E\langle\phi^\sharp, \psi^\sharp\rangle_\HH 
&=&\E\langle\phi^\sharp, \psi^\sharp - \nabla CO(\psi)\rangle_\HH\\
&=& \E\langle\phi^\sharp, - 2\, \pnabla[*][P_{\V^{(2)}_1} (1-\pathR)(d^1\phi)^\sharp]\rangle_\HH\\
&=& \E  \langle - 2\,\pnabla[\sharp] \phi^\sharp, P_{\V^{(2)}_1} (1-\pathR)(d^1\psi)^\sharp\rangle_{\HH}\\
&=& \E \langle \,( P_{\V^{(2)}_2}\tau - P_{\V^{(2)}_1} )\pnabla[\sharp] \phi^\sharp, (1-\pathR)(d^1\psi)^\sharp\rangle_{\wedge^2\HH}\\
&=& \E \langle d^{1*\sharp}(1+Q) (P_{\V^{(2)}_2}\tau - P_{\V^{(2)}_1} )\pnabla[\sharp] \phi^\sharp, \psi^\sharp\rangle_{\HH}, 
\end{eqnarray*}
where we used (\ref{eq:skew-symmetrisation}) in the skew-symmetrisation in the fourth line.
\end{proof}
We now have our main result.
\begin{theorem}
For any $L^2$ $H$-one-form $\phi$ on $C_{x_0}$, there is an $f\in \D^{2,1}$ with $\phi=df$ if and only if $d^1\phi=0$.  
Moreover, the images of $d^1$ and $d^{1*}$ are closed, and we have the following Hodge decomposition for $\phi\in L^2\Gamma\HH^*$:
\begin{equation}
\label{eq:Hodge_decomp}
\phi=df+d^{1*}\psi,
\end{equation}
for some $f\in \D^{2,1}$ and $\psi\in \Dom(d^{1*})$.
\end{theorem}
\begin{proof}
The first assertion is immediate from Theorem \ref{th:mnfd_CO_1}.
By Corollary \ref{cor:mnfd_CO_bound}, 
$\|\phi\|_{L^2\Gamma\HH^*} \le \sqrt{2}\|d^1\phi\|_{L^2\Gamma(\HH^{(2)*})}$  for all $\phi\in \Dom(d^1|_{\ker(d^1)^\perp})$. Since  $d^1|_{\ker(d^1)^\perp}$ is closed, the injective operator $d^1|_{\ker(d^1)^\perp}$ has closed range. 
From this it follows automatically that $d^{1*}$ has closed range.
\par
Vanishing of harmonic $H$-one-forms follows from the first assertion, and the Hodge decomposition given in \cite{elworthy2000special, elworthy2008l2} now has the form (\ref{eq:Hodge_decomp}).
\end{proof}
Different conventions for the definition of the inner products on $\wedge^2H$ lead to the corresponding adjoints  $d^*$ and $d^{1*}$ differing by constant multiples. This has a knock-on effect for the operator $d^{1*}d^1+dd^*$.  With our conventions in finite dimensions, this operator would not have the usual Weitzenb\"ock formula; that formula would apply to the operator $2d^{1*}d^1+dd^*$. It would, therefore, be reasonable to consider the latter as the Hodge-Kodaira Laplacian. The kernels of the two operators are the same  since both consist of forms satisfying $d^*\phi=0$ and $d^1\phi=0$, but the two operators can be expected to have different spectrums.  In addition, 
\begin{corollary}
Each of the ``Hodge-Kodaira Laplacians",  $d^{1*}d^1+dd^*$ and $2d^{1*}d^1+dd^*$,  has closed range and a spectral gap.
\end{corollary}
\begin{proof}
This is a well-known consequence of the fact that $d$ and $d^1$ have closed range. For example see Zucker \cite{zucker1979hodge} 
or Donnelly \cite{donnelly1981differential}.
\end{proof}
\begin{remark}
Theorem \ref{th:mnfd_CO_1} gives a  decomposition of $\Dom(d^1)$ into the sum of the image of $d$ (equivalently, the kernel of $d^1$) and forms $\psi$ in the domain of $d^1$ such that $\psi^\sharp\in \V^{\perp}$, the subspace  of  $L^2\Gamma \HH$ orthogonal to the  progressively measurable vector fields. 
This decomposition is unique, since if $\phi=df$ for some $f\in \D^{2,1}$ and $\phi^\sharp\in \V^{\perp}$, by the Clark-Ocone formula for functions we see $f$ must be constant, and so $\phi=0$.
\par
However, this decomposition does \emph{not} extend over all of $L^2\Gamma \HH^*$, since the projection on the first component, $\phi\mapsto d CO(\phi)$, does not extend continuously over $L^2\Gamma \HH^*$. This can be seen by taking a sequence of functions in $\D^{2,1}$ converging in $L^2$ to a function not in $\D^{2,1}$ and considering their  integral representations. 
Therefore, the map $\phi\mapsto P_{\V^{(2)}_1} (1-\pathR)(d^1\phi)^\sharp$ 
does not extend continuously over all of $L^2$, unlike the map of functions $f\mapsto P_\V \nabla f$. However, it does extend continuously over $\V^{\perp}$. 
\par
Such decompositions for higher order forms on the classical Wiener space, and ``dual"  decompositions, are given by Yang in \cite{yang2011thesis, yang2011generalised}. As pointed out in \cite{yang2011generalised}, these also determine decompositions for forms on based path spaces over compact Lie groups by the earlier results of Fang and Franchi \cite{fang1997differentiable}, when flat invariant connections are used to define the Bismut tangent spaces.
\end{remark}
\section{The pullback property of CO}
For completeness, we include the following result from \cite{yang2011thesis} concerning the pullback of  $\phi\mapsto CO(\phi)$ under It\^o maps.  An analogous result using the maps of path spaces induced by the projections  $p:K\to K/G$ of  Riemannian symmetric spaces is given in \cite{elworthy2010stochastic}. 
We use the set-up of Section \ref{sec:Ito map}, and in particular the It\^o map $\I:C_0\to C_{x_0}$ of an SDE on $M$ that induces the Levi-Civita connection of $M$. The pullback of $H$-one-forms by such an It\^o map is known to exist \cite{elworthy2007ito}.
\begin{theorem}
For any  $L^2$ $H$-one-form $\phi$ on $C_{x_0}$, we have almost surely
\begin{equation}
\label{eq:pullback}
CO(\I^*\phi)= CO(\phi)\circ \I.
\end{equation}
\end{theorem}
\begin{proof}
As both sides of (\ref{eq:pullback}) have zero expectation, it suffices to test against all functions of the form $\int_0^T\langle\dot{a}_t, dB_t\rangle_{\R^m}$, for an adapted and bounded $H$-vector field $a$ on $C_0$. As before,  $\{B_t\}_{0\le t\le T}$ denotes the canonical Brownian motion on $\R^m$. Recall that we let $\{\sigma_t\}_{0\le t\le T}$ denote both the canonical process on  $M$ and a generic element of $C_{x_0}$, so $x_t:=\sigma_t\circ \I=\I_t$
is the solution of our SDE  (\ref{eq:SDE}).  Also $\F^{x_0}_.$ refers to the natural filtration on $C_{x_0}$ and $\F^{\I}_.$ to that generated by $\I$ on $C_0$.
\par
Let $\phi$ be an $L^2$ $H$-one-form on $C_{x_0}$. We wish to show 
\begin{equation}
\label{eq:pullback2} 
\left[\int_0^T\left\langle \E(\frac{\D}{dt}\phi^{\sharp}_t|\F_t^{x_0}), d\{\sigma\}_t\right\rangle_{\sigma_t}\right]\circ \I 
= \int_0^T\left\langle  \E[\frac{d}{dt}(\I^*\phi)^{\sharp}_t|\F_t], dB_t\right\rangle_{\R^m}.
\end{equation}
Testing the left hand side of equation (\ref{eq:pullback2}), we get
\begin{eqnarray*}
& &
\int_{C_0} \int_0^T\langle  \dot{a}_t, dB_t\rangle_{\R^m}\left[\int_0^T\left\langle  \E(\frac{\D}{dt}\phi^{\sharp}_t|\F_t^{x_0}), d\{\sigma\}_t\right\rangle_{\sigma_t}\right]\circ \I\,d\gamma \\
&=&
\int_{C_0}\int_0^T\langle  \dot{a}_t,  dB_t \rangle_{\R^m}\int_0^T\left\langle  \E(\frac{\D}{dt}\phi^{\sharp}_t\circ \I|\F_t^\I), X(x_t)dB_t\right\rangle_{\R^m} d\gamma \\
&=&
\int_{C_0}\int_0^T\left\langle \E(\dot{a}_t|\F_t^\I), Y_{x_t}\E(\frac{\D}{dt}\phi^{\sharp}_t\circ \I|\F_t^\I)\right\rangle_{\R^m}dt\,d\gamma\\
&=&
\int_{C_{x_0}}\int_0^T\left\langle   X(\sigma_t)\E(\dot{a}_t|\I_t=\sigma_t), \frac{\D}{dt}\phi^{\sharp}_t(\sigma)\right\rangle_{\sigma_t} dt\,d\mu_{x_0}(\sigma)\\
&=&
\int_{C_{x_0}}\phi_\sigma\left(\overline{T\I_\sigma}[\E({a}_.|\I=\sigma)]\right)d\mu_{x_0}(\sigma),
\end{eqnarray*}
where we used the It\^o isometry in the third line, and 
equation (\ref{eq:barTI}) in the last.  
Applying the It\^o isometry again, we obtain from the right hand side
\begin{eqnarray*}
& &
\int_{C_0}\int_0^T\langle  \dot{a}_t, dB_t\rangle_{\R^m}\int_0^T\left\langle  \E[\frac{d}{dt}(\I^*\phi)^{\sharp}_t|\F_t], dB_t\right\rangle_{\R^m}d\gamma\\
&=&
\int_{C_0}\int_0^T\left\langle  \dot{a}_t, \E[\frac{d}{dt}(\I^*\phi)^{\sharp}_t|\F_t]\right\rangle_{\R^m}dt\,d\gamma\\
&=&
\int_{C_0}\langle  a, (\I^*\phi)^{\sharp}\rangle_Hd\gamma\\
&=&
\int_{C_0} \phi[\overline{T\I(a)}]\circ\I \,d\gamma\\
&=&
\int_{C_{x_0}}\phi_\sigma[\overline{T\I(a)}_\sigma]d\mu_{x_0}(\sigma)\\
&=&
\int_{C_{x_0}}\phi_\sigma\left(\overline{T\I_\sigma}[\E(a_.|\I=\sigma)]\right)d\mu_{x_0}(\sigma),
\end{eqnarray*}
where the fourth line follows from Corollary 3.7 of \cite{elworthy2007ito}, and the last from 
equation (\ref{eq:2TI}) and the adaptedness of $\{a_t\}_{t\in [0,T]}$.
\end{proof}
\bibliographystyle{amsplain} 
\bibliography{../../wt}
\end{document}